\documentclass[10pt,reqno]{amsart}

\usepackage{pdfpages}
\usepackage{pdflscape} % for 'landscape' environment

\usepackage{booktabs}
\usepackage{amssymb}
\usepackage{graphicx}
\usepackage{epstopdf}
\usepackage[hyphens]{url}
\usepackage[square,numbers,sort&compress]{natbib}

\usepackage{multirow}
\def\block(#1,#2)#3{\multicolumn{#2}{c}{\multirow{#1}{*}{$ #3 $}}}

\usepackage{tikz}
\definecolor{mediumspringgreen}{rgb}{0.0, 0.98039215, 0.60392156}

\usepackage{pgfplots} %% This Clashes with other things

\usetikzlibrary{positioning}

\def\visible<#1>{}  % beamer command not needed here

\usepackage[utf8]{inputenc}

\usepackage[english]{babel}
\usepackage{amsfonts}
\usepackage{amsmath}
\usepackage{latexsym}
\usepackage{color}
\usepackage{subfigure}
\usepackage{enumerate}

\usepackage{hyperref}

\usepackage{ifpdf}
\usepackage{listings}
\newcommand{\vect}[1]{\left(\begin{array}{@{}c@{}}#1\end{array}\right)}

\newcommand{\old}[1]{{}}

\newcommand{\bb}{\mathbb}

\newcommand{\R}{\bb R}

\newcommand{\N}{\bb N}

\newcommand{\floor}[1]{\lfloor#1\rfloor}

\newcommand\st{\mid}

\def\ve#1{\mathchoice{\mbox{\boldmath$\displaystyle\bf#1$}}
{\mbox{\boldmath$\textstyle\bf#1$}}
{\mbox{\boldmath$\scriptstyle\bf#1$}}
{\mbox{\boldmath$\scriptscriptstyle\bf#1$}}}

\newcommand{\Rr}{\mathbb{R}}

\newcommand{\cf}{\mathfrak{c}}

\newcommand{\x}{{\ve x}}

\newcommand{\w}{{\ve w}}
\renewcommand{\u}{{\ve u}}
\renewcommand{\v}{{\ve v}}

\newenvironment{psmallmatrixbig}{\bigl(\smallmatrix}{\endsmallmatrix\bigr)}
\newcommand\InlineFrac[2]{#1/#2}  % only works for simple arguments
\newcommand\ColVec[3][\relax]% Optional argument is denominator.
{
  \ifx#1\relax
  \bgroup\let\frac=\InlineFrac\begin{psmallmatrixbig}#2\vphantom{/}\\#3\vphantom{/}\end{psmallmatrixbig}\egroup
  \else
  \bgroup\let\frac=\InlineFrac\begin{psmallmatrixbig}\ifx#200\else#2/#1\fi\\\ifx#300\else#3/#1\fi\end{psmallmatrixbig}\egroup
  \fi
}

\newtheorem{theorem}{Theorem}[section]

\makeatletter
\newcommand\MkNewTheorem[2]{%
  \newtheorem{#1}{#2}[section]
  \expandafter\def\csname c@#1\endcsname{\c@theorem}
  \expandafter\def\csname p@#1\endcsname{\p@theorem}
  \expandafter\def\csname the#1\endcsname{\thetheorem}
  \expandafter\def\csname #1name\endcsname{#2}
}

\MkNewTheorem{corollary}{Corollary}
\MkNewTheorem{lemma}{Lemma}
\MkNewTheorem{proposition}{Proposition}
\MkNewTheorem{prop}{Proposition}
\MkNewTheorem{claim}{Claim}
\MkNewTheorem{observation}{Observation}
\MkNewTheorem{obs}{Observation}
\MkNewTheorem{conjecture}{Conjecture}
\MkNewTheorem{openquestion}{Open question}
\MkNewTheorem{question}{Question}

\theoremstyle{definition}
\MkNewTheorem{example}{Example}
\MkNewTheorem{exercise}{Exercise}
\MkNewTheorem{notation}{Notation}
\MkNewTheorem{assumption}{Assumption}
\MkNewTheorem{definition}{Definition}
\MkNewTheorem{remark}{Remark}
\MkNewTheorem{goal}{Goal}
\MkNewTheorem{problem}{Problem}

\makeatletter
\let\OurMathBbAux=\mathbb
\DeclareRobustCommand\OurMathBb{\OurMathBbAux}
\let\mathbb=\OurMathBb
\let\bfseries=\undefined
\DeclareRobustCommand\bfseries
{\not@math@alphabet\bfseries\mathbf
  \boldmath\fontseries\bfdefault\selectfont\let\OurMathBbAux=\mathbf}
\def\@thm#1#2#3{%
  \ifhmode\unskip\unskip\par\fi
  \normalfont
  \trivlist
  \let\thmheadnl\relax
  \let\thm@swap\@gobble
  \thm@notefont{\fontseries\mddefault\upshape\unboldmath}%   %%% <--- Added \unboldmath
  \thm@headpunct{.}% add period after heading
  \thm@headsep 5\p@ plus\p@ minus\p@\relax
  \thm@space@setup
  #1% style overrides
  \@topsep \thm@preskip               % used by thm head
  \@topsepadd \thm@postskip           % used by \@endparenv
  \def\@tempa{#2}\ifx\@empty\@tempa
    \def\@tempa{\@oparg{\@begintheorem{#3}{}}[]}%
  \else
    \refstepcounter{#2}%
    \def\@tempa{\@oparg{\@begintheorem{#3}{\csname the#2\endcsname}}[]}%
  \fi
  \@tempa
}
\makeatother

\makeatletter
\renewcommand{\pod}[1]%% a helper command used in \pmod.
{\allowbreak\mathchoice{\mkern18mu}{\mkern8mu}{\mkern8mu}{\mkern8mu}(#1)}
\makeatother

\let\epsilon=\varepsilon

\let\Myunderscore=\textunderscore   %%%% Better with textsf
\newcommand\underscore{\Myunderscore\allowbreak}
\let\_=\underscore

\title{Solid angle measure of polyhedral cones} %Computing ... via decomposition and power series}

\thanks{The authors gratefully acknowledge partial support from the National Science
  Foundation through grant DMS-2012429, awarded to Y.~Zhou.}

\author{Allison Fitisone}
\address{Allison Fitisone: Dept.\ of Mathematics, University of Kentucky}
\email{allison.fitisone@uky.edu }

\author{Yuan Zhou}
\address{Yuan Zhou: Dept.\ of Mathematics, University of Kentucky}
\email{yuan.zhou@uky.edu}
\usepackage[normalem]{ulem}

\newcommand\Figure[2][\relax]{%
  \begin{figure}[h!]
    \includegraphics[width=.8\textwidth]{#2}
    \caption{\ifx#1\relax#2\else#1\fi}
  \end{figure}
}
\newcommand*\elide{\textup{[\,\dots]}}

\begin{document}
 \newcommand{\tgreen}[1]{\textsf{\textcolor {ForestGreen} {#1}}}
 \newcommand{\tred}[1]{\texttt{\textcolor {red} {#1}}}
 \newcommand{\tblue}[1]{\textcolor {blue} {#1}}

%\makeatletter
%\newcommand{\leqnomode}{\tagsleft@true}
%\newcommand{\reqnomode}{\tagsleft@false}
%\makeatother
%% \newcommand{\tgreen}[1]{#1}
%% \newcommand{\tred}[1]{}
%% \newcommand{\tblue}[1]{#1}

\begin{abstract}

This paper addresses the computation of normalized solid angle measure of polyhedral cones. 
This is well understood in dimensions two and three. For higher dimensions, assuming that a positive-definite criterion is met, the measure can be computed via a multivariable hypergeometric series. We present two decompositions of full-dimensional simplicial cones into finite families of cones satisfying the positive-definite criterion, enabling the use of the hypergeometric series to compute the solid angle measure of any polyhedral cone. Additionally, our second decomposition method yields cones with a special tridiagonal structure, reducing the number of required coordinates for the hypergeometric series formula. Furthermore, we investigate the convergence of the hypergeometric series for this case. Our findings provide a powerful tool for computing solid angle measures in high-dimensional spaces. 

\end{abstract}
\maketitle

%\clearpage
%% {\footnotesize
%% \tableofcontents}
%%%%%%%%%%%%%%%%%%%%%%%%%%%%%%%%%%%%%%%%

\section{Introduction}
The study of solid angles %in higher dimensions 
has gained a lot of traction in the past several years and has been investigated in several publications (see \cite{ribando2006measuring}, \cite{Beck2007}, \cite{Beck2009PositivityTF}, \cite{fukshansky2011bounds}, \cite{mazonka2012solid} etc.). 
%\tred{(why are these three papers cited here, and in this order?)}
%most cited + more recent
Of particular interest is the solid angle measure of a polyhedral cone in higher dimensions. This measure has widespread potential applications, from computing the expected number of simplices in a triangulation of an $n$-cube \cite{ribando2000thesis},
 %\tred{(right citation?)} 
 to calculating relative pixel purity index (PPI) scores \cite{heylen2012multidimensional}, to computing the volume of the feasibility domain of an ecological community \cite{mariani2019nestedness}. The various ways in which higher-dimensional solid angle measures appear in literature demonstrate the importance of the topic and necessitate a deeper understanding of it. 

Recall that a polyhedral cone is the intersection of a finite number of half-spaces passing through the origin. By the well-known Minkowski–Weyl theorem, a polyhedral cone is also the convex conic hull of a finite number of vectors. 
In $\R^2$, the solid (or plane) angle measure can be computed via the standard inner product. 
In $\R^3$, the solid angle measure of a cone generated by three unit vectors is the area of the spherical triangle on the unit sphere formed by the unit vectors. There is a closed formula dating back to Euler and Lagrange which uses the scalar triple product. Assume that ${\v_1}, {\v_2}$ and ${\v_3}$ are unit vectors in $\R^3$. The solid angle measure of the cone generated by ${\v_1}, {\v_2}$ and ${\v_3}$ is given by \cite{eriksson1990measure}:
\[
E := 2\tan^{-1}\left(\frac{|{\v_1} \cdot ({\v_2} \times {\v_3})|}{1+{\v_2}\cdot {\v_3}+{\v_2}\cdot {\v_1}+{\v_1}\cdot {\v_3}}\right)
\; ,\]
and the normalized solid angle measure of the cone, which is the proportion of $\R^3$ that the cone occupies, is therefore $\frac{E}{4\pi}$.

In higher dimensions, we define the solid angle measure of a polyhedral  cone according to \cite{hajja2002measure}, and we normalize it in such a way that the whole space has measure $1$, according to \cite{ribando2006measuring}.
\begin{definition} \label{def:solid-angle} %Let $C\subseteq \R^n$ be a polyhedral cone. 
The \emph{normalized solid angle measure}, of a polyhedral cone $C\subseteq \R^n$ with respect to the space $\Rr^n$, denoted by $\tilde{\Omega}_n(C)$, is defined as the ratio of the $(n-1)$-dimensional volume of the intersection of the cone with the unit sphere $S_{n-1}$ in $\R^n$ centered around the origin to the volume of the unit sphere, i.e.,
\begin{equation} \label{eq:solid-angle-as-ratio-of-volumes}
  \tilde{\Omega}_n(C) = \frac{\text{vol}_{n-1}(C\cap S_{n-1})}{\text{vol}_{n-1}(S_{n-1})}.
\end{equation}
\end{definition}
We note that the above definition works for convex non-polyhedral cones as well. It also matches with the definition in lower dimensions $n=2$ and $n=3$. When the cone $C$ is of dimension $d<n$, let $S_{d-1}$ denote the unit sphere centered around the origin in the $d$-dimensional linear space containing $C$. We define the normalized solid angle measure of $C$ with respect to the linear span of $C$ in a similar manner: %\tblue{TODO: a bit funky because $S_{d-1}$ is not really defined.}
\begin{equation*} %\label{eq:solid-angle-dim-d}
  \tilde{\Omega}_d(C) = \frac{\text{vol}_{d-1}(C\cap S_{d-1})}{\text{vol}_{d-1}(S_{d-1})}.
\end{equation*}

\begin{remark} \label{rmk:simplicial_and_fulldimensonal}
Without loss of generality, in the following, we need only consider \emph{full-dimensional simplicial cones} $C$. Indeed, if $C$ is not simplicial, then triangulation allows us to decompose $C$ into simplicial cones (i.e. having a linearly independent system of generators), and $\tilde{\Omega}_n(C)$ is the sum of the normalized solid angles of the cones in the decomposition. If $C$ is not full-dimensional, then by our normalization, $\tilde{\Omega}_n(C)=0.$ Note however that $\tilde{\Omega}_d(C)\neq 0$, where $d$ is the dimension of the linear span of $C$.  
%Suppose a polyhedral cone $C$ is generated by exactly $m$ extreme vectors, $\v_1, \hdots, \v_m$, and the linear span of $C$ has dimension $n$. If $m>n$,  Let $d$ be the dimension of the ambient space, i.e. the space where $\v_1, \hdots, \v_m$ live. If $n<d$, 
Other papers normalize with respect to a half space \cite{gourion2010deterministic} or the affine space \cite{kabluchko2020absorption}. We show the distinction in the following example.
\end{remark}

\begin{example}
Consider the polyhedral cone $C\subseteq \Rr^3$ generated by the vectors $\v_1 = (1,0,0)$, $\v_2 = (0,1,0)$ and $\v_3 = (-1,0,0)$. The ambient space has dimension three. However, it is clear that $C$ lies on the plane $z=0$, and so the dimension of the linear span of $C$ is 2. Thus, $\tilde{\Omega}_3(C)=0.$ However, when the measure is considered with respect to the linear span, we see $\tilde{\Omega}_2(C) \neq 0$. Note that although $C$ is not simplicial, it can be triangulated into two simplical cones: $C_1$ generated by the vectors $\v_1$ and $\v_2$, and $C_2$ generated by the vectors $\v_2$ and $\v_3$. Thus, $\tilde{\Omega}_2(C)=\tilde{\Omega}_2(C_1)+\tilde{\Omega}_2(C_2)=0.25+0.25=0.5.$  
\end{example}

To the best of our knowledge, there currently does not exist a neat closed formula for the solid angle measure of polyhedral cones in dimensions greater than three. 
%\tred{write in words.}
As noted by Seeger and Torki \cite{seeger2015centers}, the derivation of an easily computable formula for the solid angle measure is a difficult task and one must use other methods for approximation. These include probabilistic methods, numerical integration methods, and the multi-variate power series method. In regard to the first two methods, see \cite{gourion2010deterministic}. 
%\tred{both the two first methods have the citation [10] and the last one should have [16] and [3]? Then rephrase the sentence: In regards to the first two methods, see [10] blahblahblah.}
In regard to the last method, Ribando \cite{ribando2006measuring} and Aomoto \cite{aomoto1977analytic} independently discovered a formula for the normalized solid angle measure of a simplicial cone, in the form of a multi-variate hypergeometric series.

Ribando derived the power series formula by noting that there are many equivalent ways to represent the normalized solid angle measure. It is natural to view the volumes in \autoref{def:solid-angle} as integrals. Doing so gives the following equivalent definitions of $\tilde{\Omega}_n$:
\begin{equation} \label{eq:solid-angle-proportion}
    \tilde{\Omega}_n(C) = \frac{\text{vol}_{n-1}(C\cap S_{n-1})}{\text{vol}_{n-1}(S_{n-1})}  = \frac{\int_{C} f(\x) d\x}{\int_{\Rr^n} f(\x) d\x} \; ,
\end{equation}
where $f\colon \R^n \to \R$ is any function that is invariant under rotations around the origin. Ribando determined that the form $f=e^{-\|\x\|^2}$ led to nice computations with the integral
\begin{equation} \label{eq:solid-angle-exp}
    \tilde{\Omega}_n(C) = \frac{\int_{C} e^{-\|\x\|^2} d\x}{\int_{\Rr^n} e^{-\|\x\|^2} d\x} = \frac{\int_{C} e^{-\|\x\|^2} d\x}{\pi^{\frac{n}{2}}} \; ,
\end{equation}
which when simplified and expanded gave rise to a convergent hypergeometric series. Furthermore, Ribando showed that the convergence of the power series was dependent on a certain matrix, which we call the \emph{associated matrix}.

Ribando's results are presented in \autoref{thm:ribando2006} below.
Recall that by \autoref{rmk:simplicial_and_fulldimensonal}, we will only consider full-dimensional simplicial cones.
\begin{definition}
\label{def:associated-matrix}
 Let $C$ be a simplicial cone in $\R^n$, whose generators are the linearly independent unit vectors ${ \v_1}, \dots, {\v_n} \in \R^n$. Then, the \emph{associated matrix} of $C$, denoted by $M_n(C)$, is the $n\times n$ matrix whose $(i,j)$-th entry is $-|\v_i \cdot \v_j|$ for $i\neq j$, and whose diagonal $(i,i)$-th entry is 1.
\begin{equation} \label{eq:associated-matrix}
M_n(C) = \begin{bmatrix}
1 & -|\v_1 \cdot \v_2| & \cdots & -|\v_1 \cdot \v_n| \\
-|\v_2 \cdot \v_1| & 1 & \cdots & -|\v_2 \cdot \v_n| \\
\vdots & \vdots & \ddots & \vdots \\
-|\v_n \cdot \v_1| & -|\v_n \cdot \v_2| & \cdots & 1 \\
\end{bmatrix}.
\end{equation}
 
\end{definition}

\begin{theorem}[{\cite[Theorem 2.2 and Corollary 3.3]{ribando2006measuring}, rephrased}] \label{thm:ribando2006}
Let $C \subseteq \Rr^n$ be the simplicial cone generated by the unit vectors ${\v_1}, \dots, {\v_n}$. Let $V \in \R^{n\times n}$ 
be the matrix whose $i^{th}$ column is ${\v_i}.$ Let $\alpha_{ij}={\v_i}\cdot {\v_j}$ for $1\leq i, j \leq n$.  Let 
\begin{equation} \label{eq:t-alpha-series}
    T_{\ve \alpha} = \frac{|\det V|}{(4\pi)^{\frac{n}{2}}} \sum_{{\ve a}\in \N^{\binom{n}{2}}} \left[ \frac{(-2)^{\sum_{i<j} a_{ij}}}{\prod_{i<j} a_{ij}!} \prod_{i=1}^{n} \Gamma\left(\frac{1+\sum_{m\neq i} a_{im}}{2}\right)\right] {\ve \alpha}^{\ve a}.
\end{equation}
Then, $T_{\ve \alpha}$ converges absolutely to the normalized solid angle measure $\tilde{\Omega}_n(C)$ of the cone if and only if its associated matrix $M_{n}(C)$ is positive definite.
\end{theorem}

\noindent In the above series~\eqref{eq:t-alpha-series}, $\Gamma$ is the Euler-Gamma function; 
$\ve{\alpha} = (\alpha_{12}, \alpha_{13}, \dots \alpha_{n-1,n})$ is a multivariable in $\binom{n}{2}$ variables; and $\ve{a}=(a_{12}, a_{13}, \dots, a_{n-1,n})$ is a multiexponent. We define \[\ve{\alpha}^{\ve{a}} := \prod_{i<j}^n \alpha_{ij}^{a_{ij}}.\] When $i>j$, we set $a_{ij} = a_{ji}$, and we define $\sum_{m\neq i} a_{im}$ to be the sum over all the terms in the multiexponent where $i$ appears as either the first or second index. 

Other functionals $f$ have been used in equation~\eqref{eq:solid-angle-proportion} to represent high-dimensional solid angle measures as integrals. 
Hajja and Walker \cite[Theorem 1]{hajja2002measure} use $f(\x)=1$ and consider a standard change of variables, to obtain an integral formula for the solid angle measures, where the integral is taken over the portion of the unit sphere in the positive orthant.
%to obtain an equivalent formula, which after having applied our normalization gives:
%\[\tilde{\Omega}_n(C)= \frac{\Gamma(\frac{n}{2}+1)}{n\pi^{\frac{n}{2}}}|\det A| \int_S \| Av \|^{-n} dS\]
%where the integral is over the portion of the unit sphere in the positive orthant, and where $A$ is the $n\times n$ matrix whose columns correspond to $n$ linearly independent vectors that generate the cone $C\subseteq\Rr^n$. 
Kabluchko and Zaporozhets \cite[Proposition 1.1]{kabluchko2020absorption}  use a form related to Gaussian distribution to present an equivalent formula for the solid angle measures of a specific class of cones. 
For the purposes of this paper, we will focus exclusively on Ribando's hypergeometric series formula.

As aforementioned, alternative methods exist for estimating solid angle measures. Numerical integration methods often give rise to rough estimates of solid angle measures, or involve integrals that can only be computed under specific conditions \cite{gourion2010deterministic}. Random estimation techniques require large samples for accurate approximations, as discussed in \cite{gourion2010deterministic}. Heylen and Scheunders \cite{heylen2012multidimensional} use a multi-dimensional extension of Monte--Carlo sampling to obtain estimates of solid angle measures. Adam and Speciel \cite{adam2020evaluating} re-express the solid angle measure as in equation~\eqref{eq:solid-angle-proportion}, using the Heaviside step function followed by a Monte--Carlo approach for estimation. The multi-variate power series method is arguably the method of choice for computing exact high dimensional solid angle measures \cite{adam2020evaluating}. However, the power series method has previously been deemed as too complex and technical \cite{fukshansky2021positive}, \cite{aljundi2019online}.

 Ribando notes that there are two significant issues in applying the power series method \cite{ribando2006measuring}. The first major issue is that in higher dimensions, computational feasibility is hindered by the large number of coordinates needed to use the formula. To compute the normalized solid angle measure of an $n-$dimensional simplicial cone, one needs $\binom{n}{2}$ coordinates.  Several authors make reference to the power series, but often cite it being computationally untractable for high dimensions.  The second issue is that the positive definite-ness of the associated matrix is an essential assumption for applying the power-series formula. When this criterion is not met, ${\ve \alpha}$ lies outside of the domain of convergence of the series, and the formula cannot be used in a way that is meaningful.

In this paper, we present a method that allows one to compute the normalized solid angle measure of any polyhedral cone via the power series method. Moreover, we address the two major issues standing in the way of widespread use of the power series method as a means to compute solid angle measures. 

We address the positive-definite-ness criterion via signed decompositions of cones. %, where we consider the indicator functions of cones. 
In Section 2, we investigate Brion--Vergne decomposition of cones with respect to hyperplanes. We demonstrate in \autoref{thm:decomp-1} and \autoref{cor:decomp-1} that Brion--Vergne decomposition of a simplicial cone with respect to a particular hyperplane results in finitely many cones that either have a positive definite associated matrix (and so their solid angle measures can be computed directly via the power series formula), or contain lines (and so by applying \autoref{cor:project-out-linear-subspaces}, the computation of solid angle measures can be simplified to a lower-dimensional problem, which can be addressed using induction on dimension). 
In Section 3, we consider Brion--Vergne decomposition of cones with respect to lines. We demonstrate in \autoref{thm:decomp-2} that Brion--Vergne decomposition with respect to a particular line leads to a second decomposition of a simplicial cone into finitely many cones, each of which has a positive definite associated matrix. This theorem gives a method to decompose solid angles lying outside the domain of convergence into solid angles lying within the domain of convergence of the series in \autoref{thm:ribando2006}. 

We also address the computational feasibility of the power series method by reducing the number of coordinates needed. 
In Section 4, we explore the properties of cones resulting from the application of a stronger version of \autoref{thm:decomp-2}. These cones have associated matrices that are not only positive-definite but also tridiagonal, reducing the number of required coordinates for the power series formula from $\binom{n}{2}$ to $n-1$. \autoref{thm:tridiag-matrix} shows that under the tridiagonal assumption, the power series formula always converges to the solid angle measure. In \autoref{thm:error-any-dim}, we examine the asymptotic error of the power series formula %~\eqref{eq:t-beta-series}
for cones with tridiagonal associated matrices.

In addition, we have developed software that can compute the normalized solid angle measure of polyhedral cones in any dimension using decomposition methods presented in the current paper. Our software is written in SageMath \cite{sagemath}, a comprehensive Python-based open source computer algebra system. It is available at \url{https://github.com/yuan-zhou/solid-angle-code}, and can be run on a local installation of SageMath or online using SageMath on CoCalc. However, it is important to note that this software is not the focus of the present paper. We will defer all computational questions and experiments to a forthcoming paper.

\section{First decomposition method} %with cones containing lines}
\label{sec:decomp-cones-containing-line}
\subsection{Brion--Vergne decomposition with respect to a hyperplane}
We will make use of Brion--Vergne (B--V) decomposition (see \cite{brion1997residue}, \cite{baldoni2013intermediate}) with respect to a hyperplane. B--V decomposition is a signed decomposition of a simplicial cone into a finite family of full-dimensional simplicial cones. The hyperplane determines a facet for each of the cones in the decomposition. We adopt some of the notation used in \cite{baldoni2013intermediate} and reformulate the decomposition below. 

\begin{definition} \label{def:notations}
Let $\w_1, \dots, \w_d$ be vectors in $\Rr^n$. We denote by $\langle \w_1, \dots, \w_{d} \rangle$ the linear span of $\w_1, \dots, \w_d$.  We denote by $\cf(\w_1, \dots, \w_d)$ the cone generated by the vectors $\w_1, \dots, \w_d$. 

Let $L\subseteq \Rr^n$ be a hyperplane, and $\w_i \in \Rr^n\setminus L$. We denote by $\rho_{i}\colon \R^n \to L$ the projection onto $L$, parallel to $\w_i$.

Let $C$ be a cone. 
We denote by $[C]$ the indicator function of the cone $C$.

In this paper, we say that a cone $C$ can be decomposed into cones $C_1, \dots, C_k$, if their indicator functions satisfy the relation 
\begin{equation*}
[C] \equiv \sum_{i=1}^k s_i[C_i], \text{ where } s_i \in \{\pm 1\},
\end{equation*}
modulo indicator functions of cones containing lines (in \autoref{sec:decomp-cones-containing-line}) or modulo indicator functions of lower-dimensional cones (in \autoref{sec:decomp-lower-dim-cones}).
\end{definition}

\begin{theorem}[{Brion--Vergne decomposition with respect to a hyperplane} {\cite[Proposition 15b]{baldoni2013intermediate}}]
 \label{thm:BV-wrt-hyperplane}
Let $\w_1, \dots, \w_n$ form a basis of $\Rr^n$ and let cone $C =\cf(\w_1, \dots, \w_n)$. Let $L\subseteq \Rr^n$ be a hyperplane. Assume that $\w_i \in L$ if and only if $r+1 \leq i\leq s$ and that $\w_i$ lie on one side of $L$ for $1\leq i\leq r$, and on the other side for $s+1\leq i\leq n$. Then, we have the following relation modulo indicator functions of cones containing lines.
\begin{equation*}
[C] \equiv \sum_{i=1}^r [\Rr_+\w_i + \rho_{i}(C)] - \sum_{i=s+1}^n [\Rr_+(-\w_i) + \rho_{i}(C)].
\end{equation*}
\end{theorem}

\subsection{Solid angle decomposition that includes cones containing lines} 
We will first show that the normalized solid angle of an orthogonal sum of convex cones is the product of the normalized solid angles of the cones in the orthogonal sum.

% The previous wrong version of the lemma has been removed to side note for solid angle tex
\begin{lemma}\label{lemma:orthogonality}
Let $C\subseteq \R^n$ be a polyhedral cone (or more generally, a convex cone) of dimension $d$. Suppose that $C=C_1\oplus C_2$ is the orthogonal sum of cones $C_1$ and $C_2$ of dimensions $d_1$ and $d_2$, respectively. Then, the normalized solid angle measure of $C$ with respect to the linear span of $C$ satisfies that
\begin{equation*}
\label{eq:orthogonality}
\tilde{\Omega}_d(C)=\tilde{\Omega}_{d_1}(C_1) \cdot \tilde{\Omega}_{d_2}(C_2).
\end{equation*}
\end{lemma}
\begin{proof}
Since $C_1$ and $C_2$ are orthogonal, $d=d_1+d_2$. Any $\ve{u}\in C$ can be uniquely expressed as $\ve{u}=\ve{v}+\ve{w}$ where $\ve{v}\in C_1, \ve{w}\in C_2$ and $\|\ve{u}\|^2=\|\ve{v}\|^2+\|\ve{w}\|^2$. Thus,  
\[\int_{C} e^{-\|\ve{u}\|^2} d\ve{u} = \int_{C_1\oplus C_2} e^{-\|\ve{v}+\ve{w}\|^2} d(\ve{v}+\ve{w}) =\int_{C_1} e^{-\|\ve{v}\|^2} d\ve{v} \int_{C_2}e^{-\|\ve{w}\|^2} d\ve{w}.\]
%Since for any positive integer $n$, \[\int_{\Rr^n} e^{-\|\ve{u}\|^2} d\ve{u}=\pi^\frac{n}{2},\] 
It then follows from equation~\eqref{eq:solid-angle-exp} that 
\[
 \tilde{\Omega}_d(C)=\frac{\int_{C} e^{-\|\ve{u}\|^2} d\ve{u}}{\pi^{d/2}}=\frac{\int_{C_1} e^{-\|\ve{v}\|^2} d\ve{v}}{\pi^{d_1/2}}\cdot\frac{\int_{C_{2}}e^{-\|\ve{w}\|^2} d\ve{w}}{\pi^{d_2/2}}=\tilde{\Omega}_{d_1}(C_1) \cdot \tilde{\Omega}_{d_2}(C_2).\]
\end{proof}

In the special case that one of the orthogonal parts is an $\ell$-dimensional linear subspace $L$ of $\R^n$, since $\tilde{\Omega}_\ell(L)=1$, we enjoy the following corollary.

\begin{corollary} \label{cor:project-out-linear-subspaces}
Let $C\subseteq \R^n$ be a polyhedral cone (or more generally, a convex cone) of dimension $d$, such that $C=C' \oplus L$ is the orthogonal sum of a cone $C'$ of dimension $d'<d$ and a linear subspace $L$ of $\R^n$. Then, 
$\tilde{\Omega}_d(C) = \tilde{\Omega}_{d'}(C')$.
\end{corollary}

\autoref{cor:project-out-linear-subspaces} demonstrates how computing solid angles of cones containing lines reduces to computing solid angles of lower dimensional cones. 

We now present a signed decomposition of a given full-dimensional simplicial cone whose solid angle measure is of interest (recall the assumption is without loss of generality by \autoref{rmk:simplicial_and_fulldimensonal}), in a way that the resulting cones either contain lines so that \autoref{cor:project-out-linear-subspaces} applies, or have positive definite associated matrices so that \autoref{thm:ribando2006} applies.
 
\begin{theorem} \label{thm:decomp-1}
Given linearly independent unit vectors $\w_1, \w_2,\dots, \w_n \in \Rr^n$, the cone $\cf(\w_1, \w_2, \dots, \w_n)$ can be decomposed into a finite family of cones, each of which is either:
\begin{enumerate}[\rm(I)]
    \item \label{Form 1} a cone containing lines, or
    \item \label{Form 2} a cone $\cf(\v_1, \v_2,\dots, \v_n)$ of dimension $n$ whose associated matrix %$M_n$ %according to equation~\eqref{eq:associated-matrix}
\[M_n= \begin{bmatrix}
1 & -|\v_1 \cdot \v_2| & \cdots & -|\v_1 \cdot \v_n| \\
-|\v_2 \cdot \v_1| & 1 & \cdots & -|\v_2 \cdot \v_n| \\
\vdots & \vdots & \ddots & \vdots \\
-|\v_n \cdot \v_1| & -|\v_n \cdot \v_2| & \cdots & 1 \\
\end{bmatrix}\]
is positive definite. In particular,
    \begin{enumerate}[\rm(a)]
        \item \label{property:D1-norm-is-1} $\|\v_i\| = 1$ for $i=1, 2,\dots, n$
        \item \label{property:D1-vn=wn} $\v_n = \w_n$
        \item \label{property:D1-span} $\langle \v_1, \v_2,\dots, \v_{n-1} \rangle = \langle \w_1, \w_2,\dots, \w_{n-1} \rangle$
        \item \label{property:D1-orthogonal-to-vn} $\v_i \cdot \v_n =0$ for $i = 1, 2,\dots, n-2$
        \item \label{property:D1-assoc-matrix-pos-def} 
        $\x^T M_{n}\, \x = \left\lVert\begin{bmatrix}\epsilon_1\v_1, \dots, \epsilon_n\v_n\end{bmatrix} \x\right\rVert^2$ holds for every $\x\in \Rr^n$, where $\epsilon_i = \pm 1$.
        %There exist $\epsilon_1, \epsilon_2, \dots, \epsilon_n \in \{\pm 1\}$ such that  \[\x^T M_{n} \x = \left\lVert\begin{bmatrix}\epsilon_1\v_1, \dots, \epsilon_n\v_n\end{bmatrix} \x\right\rVert^2\] holds for every $\x\in \Rr^n$.
    \end{enumerate}
\end{enumerate}
\end{theorem}

Before proving the theorem, we point out that the property~\eqref{property:D1-assoc-matrix-pos-def} ensures the positive definiteness of the associated matrix $M_n$ given by equation~\eqref{eq:associated-matrix}. Indeed, for any nonzero vector $\x := \vect{x_1, \dots, x_n}$, since %the generators 
$\v_1, \dots, \v_n$ are linearly independent,
\[\begin{bmatrix}\epsilon_1\v_1, \dots, \epsilon_n\v_n\end{bmatrix} \x = \sum_{i=1}^n \epsilon_i x_i \v_i \neq \ve 0,\]
and hence $\x^T M_{n}\, \x = \left\lVert\begin{bmatrix}\epsilon_1\v_1, \dots, \epsilon_n\v_n\end{bmatrix} \x\right\rVert^2 > 0$. As a consequence, the solid angle formula~\eqref{eq:t-alpha-series} for a cone in the decomposition, which is of the latter form~\eqref{Form 2}, is a convergent power series. 

\begin{proof}
We proceed by induction on the dimension. 

Consider the base case $n=2$. Let $\w_1 = \vect{w_{11},w_{12}}$ and $\w_2 = \vect{w_{21}, w_{22}}$ be linearly independent unit vectors. We show that the cone $\cf(\w_1, \w_2)$ is already of the form~\eqref{Form 2}. The cone satisfies the properties~\eqref{property:D1-norm-is-1} through \eqref{property:D1-orthogonal-to-vn} trivially. Let $\x = \vect{x_1 , x_2}\in \Rr^2$. The equation in property \eqref{property:D1-assoc-matrix-pos-def} has the left-hand-side
\begin{equation*}
\x^T M_{n}\, \x =
    \begin{bmatrix} x_1 & x_2 \end{bmatrix} \cdot \begin{bmatrix} 1 & -|\w_1 \cdot \w_2| \\-|\w_1 \cdot \w_2| & 1 \end{bmatrix} \cdot \begin{bmatrix} x_1 \\ x_2 \end{bmatrix}  = {x_1}^2 + {x_2}^2 -2 x_1 x_2 |\w_1 \cdot \w_2|,
\end{equation*}
and the right-hand side
\begin{align*}
&\left\lVert\begin{bmatrix}\epsilon_1\w_1, \epsilon_2\w_2\end{bmatrix} \x\right\rVert^2 = \left\lVert\begin{bmatrix} \epsilon_1w_{11} & \epsilon_2w_{21} \\ \epsilon_1w_{12} &\epsilon_2w_{22} \end{bmatrix} \cdot \begin{bmatrix} x_1 \\ x_2\end{bmatrix}\right\rVert^2 \\
& = \left(\epsilon_1 x_1 w_{11}+\epsilon_2 x_2 w_{21}\right)^2 + \left(\epsilon_1 x_1 w_{12}+\epsilon_2 x_2 w_{22}\right)^2 \\
&= {x_1}^2(w_{11}^2+w_{12}^2) + {x_2}^2(w_{21}^2+w_{22}^2) + 2\epsilon_1 \epsilon_2 {x_1} {x_2} (w_{11}w_{21}+w_{12}w_{22})\\
&= {x_1}^2 + {x_2}^2 + 2\epsilon_1\epsilon_2 {x_1} {x_2} (\w_1\cdot \w_2).
\end{align*}
Thus, for appropriate choices of $\epsilon_1$ and $\epsilon_2$, we see that property~\eqref{property:D1-assoc-matrix-pos-def} is satisfied.

Suppose the statement is true for dimension $n-1$. Now consider the cone $\cf(\w_1, \w_2, \dots, \w_n)$ in dimension $n$ which is not already of the form~\eqref{Form 1} or \eqref{Form 2}. We distinguish two cases depending on the orthogonality.

\textbf{Case 1}: Suppose that $\w_n$ is orthogonal to all $\w_i$ for $1\leq i \leq n-1$.
Set $\v_n = \w_n$. By the inductive hypothesis, there is a (signed) decomposition of the cone $\cf(\w_1, \dots, \w_{n-1})$ into finitely many cones $\mathfrak{k}_1, \dots, \mathfrak{k}_N$
that either contain lines or are of the second form and satisfy the desired properties. Then, it is clear that the cones 
\[C_1 := \mathfrak{k}_1+\cf(\v_n), \; C_2 := \mathfrak{k}_2+\cf(\v_n), \; \dots, \; C_N := \mathfrak{k}_N+\cf(\v_n)\] 
obtained by the Minkowski sums give a decomposition of $\cf(\w_1, \dots, \w_n)$. %Rename these cones to $C_1, \dots, C_N$ respectively.

If $\mathfrak{k}_i$ is a cone containing lines in the decomposition of $\cf(\w_1, \dots, \w_{n-1})$, then $C_i$ also contains lines, and hence is of form~\eqref{Form 1}. 
Otherwise, $\mathfrak{k}_i$ must be of the form $\cf(\v_1, \dots, \v_{n-1})$ that satisfies the properties~\eqref{property:D1-norm-is-1}--\eqref{property:D1-assoc-matrix-pos-def} of dimension $n-1$. It is clear that $C_i=\cf(\v_1, \dots, \v_{n-1}, \v_n)$ satisfies properties~\eqref{property:D1-norm-is-1}, \eqref{property:D1-vn=wn} and \eqref{property:D1-orthogonal-to-vn} of dimension $n$. By the inductive hypothesis, 
$\langle \v_1, \dots, \v_{n-2} \rangle = \langle \w_1, \dots, \w_{n-2}\rangle$ and $\v_{n-1} = \w_{n-1}$. Thus, we have that $\langle \v_1, \dots, \v_{n-2}, \v_{n-1} \rangle = \langle \w_1, \dots, \w_{n-2}, \w_{n-1}\rangle$ and so property~\eqref{property:D1-span} is satisfied. It remains to verify property~\eqref{property:D1-assoc-matrix-pos-def}. 
We note that since $\v_i\cdot \v_n =0$, the associated matrix of $C_i$ according to \autoref{def:associated-matrix} satisfies
\[ M_{n}(C_i)= \begin{bmatrix} & & & 0\\ & M_{n-1}(
\mathfrak{k}_i) & & \vdots \\ & & & 0\\ 0 & \dots & 0 & 1 \end{bmatrix}. \]
By the inductive hypothesis, there exist $\epsilon_1, \dots, \epsilon_{n-1} \in \{\pm 1\}$ such that the left-hand-side of the equation in property~\eqref{property:D1-assoc-matrix-pos-def} can be written as
\begin{align*}
\x^T M_n(C_i) \x 
&=  \begin{bmatrix} x_1 \cdots x_{n-1} \end{bmatrix} M_{n-1}(\mathfrak{k}_i)\begin{bmatrix} x_1 \\ \vdots \\ x_{n-1} \end{bmatrix}  + x_n^2\\
     %& = \left\lVert\begin{bmatrix}\epsilon_1\v_1, \dots, \epsilon_{n-1}\v_{n-1}\end{bmatrix} \begin{bmatrix} x_1 \\ \vdots \\ x_{n-1} \end{bmatrix} \right\rVert^2 + x_n^2\\
     & = \left\lVert \epsilon_1x_1\v_1+\dots+\epsilon_{n-1}x_{n-1}\v_{n-1}\right\rVert^2+x_n^2 \\
     & = \left\lVert\begin{bmatrix}\epsilon_1\v_1, \dots, \epsilon_{n-1}\v_{n-1}, \epsilon_n\v_n\end{bmatrix} \x\right\rVert^2, 
\end{align*}
which is equal to its right-hand side, where we set $\epsilon_n = 1$.
Thus, property~\eqref{property:D1-assoc-matrix-pos-def} holds.

%By the inductive hypothesis, there exist $\epsilon_1, \dots, \epsilon_{n-1} \in \{\pm 1\}$ such that the left-hand-side $\x^T M_n(C_i) \x$ of the equation in property~\eqref{property:D1-assoc-matrix-pos-def} can be written as \[\begin{bmatrix} x_1 \cdots x_{n-1} \end{bmatrix} M_{n-1}(\mathfrak{k}_i)\begin{bmatrix} x_1 \\ \vdots \\ x_{n-1} \end{bmatrix}  + x_n^2 = \left\lVert \epsilon_1x_1\v_1+\dots+\epsilon_{n-1}x_{n-1}\v_{n-1}\right\rVert^2+x_n^2 \] which is equal to its right-hand side $\left\lVert\begin{bmatrix}\epsilon_1\v_1, \dots, \epsilon_{n-1}\v_{n-1}, \epsilon_n\v_n\end{bmatrix} \x\right\rVert^2$, where we set $\epsilon_n = 1$.

\smallskip 
\textbf{Case 2}: Suppose that $\w_n$ is not orthogonal to all $\w_1, \dots, \w_{n-1}$.
Without loss of generality, we may assume that $\w_{n-1} \cdot \w_n \neq 0$. 

We construct a hyperplane $L$ in order to apply Brion--Vergne decomposition \autoref{thm:BV-wrt-hyperplane} as follows. For $1\leq i\leq n-2$, we define 
\begin{equation}\label{eq:li}
\ve{\ell}_i := \w_i - \frac{\w_i \cdot \w_n}{\w_{n-1} \cdot \w_n} \w_{n-1}.
\end{equation}
We remark that $\ve{\ell}_i \cdot \w_n = 0$ for all $i=1, 2,\dots, n-2$.
Let 
\[L := \langle\ve{\ell}_1, \dots, \ve{\ell}_{n-2}, \w_n\rangle.\]
This $L$ is a hyperplane of $\R^n$, since if there are scalars $\lambda_i$'s making \[0=\sum_{i=1}^{n-2} \lambda_i \ve{\ell}_i+\lambda_n \w_n = \sum_{i=1}^{n-2} \lambda_i \w_i - \sum_{i=1}^{n-2} \lambda_i \frac{\w_i \cdot {\w_n}}{\w_{n-1} \cdot {\w_n}} \w_{n-1} + \lambda_n {\w_n},\]
then the linear independence of the $\w_i$'s implies each $\lambda_i =0$. 

Next, we decide for each $i$ whether $\w_i \in L$, and we use $s_i = -1$ or $1$
to indicate which open half-space that $\w_i$ belongs to.  It is clear that $\w_n \in L$. %We let $s_n=0$, accordingly. 
We also know that $\w_{n-1} \not\in L$, because otherwise $L$ would contain all $\w_1, \dots, \w_n$, a contradiction. We set $s_{n-1}=1$.
For $1\leq i\leq n-2$, if $\w_i\cdot \w_n=0$, then it follows from equation~\eqref{eq:li} that $\w_i = \ve{\ell}_i \in L$; %, and we let $s_i =0$; 
otherwise, $\w_i\cdot \w_n\neq 0$, then we have $\w_i \not\in L$ since
\begin{equation} \label{eq:wi}
\w_i = \ve{\ell}_i + \left(\frac{\w_i \cdot \w_n}{\w_{n-1} \cdot \w_n}\right) \w_{n-1}
\end{equation}
again by equation~\eqref{eq:li}.
We let accordingly 
\begin{equation*}
s_i =
\begin{cases}
1, & \text{if $\w_i$ and $\w_{n-1}$ are on the same side of $L$}\\
-1, & \text{if $\w_i$ and $\w_{n-1}$ are on the opposite sides of $L$.}
\end{cases}
\end{equation*}
Denote 
\[
I^+ = \{1\leq i\leq n-1 \st s_i =1\} \text{ and } I^- = \{1\leq i\leq n-1 \st s_i = -1\}.
\]
We remark that $I^- \cup I^+$ is the set of $i$ such that $\w_i \in \R^n \setminus L$, or equivalently, the set of $1\leq i\leq n-1$ such that $\w_i \cdot \w_n \neq 0$.
%\begin{align*}
%    I^+ &= \{1\leq i\leq n-1 \st s_i =1\} \\
%    I^- &= \{1\leq i\leq n-1 \st s_i = -1\} \\
%    I^\pm &= I^+\cup I^- = \{1\leq i\leq n-1 \st \w_i \cdot\w_n \neq 0\}.
%\end{align*}

By \autoref{thm:BV-wrt-hyperplane}, modulo indicator functions of cones containing lines, we can decompose the cone $C = \cf(\w_1, \w_2, \dots, \w_n)$ as
 \begin{equation*} 
    [C] \equiv \sum_{i \in I^+} [\Rr_+\w_i + \rho_{i}(C)] - \sum_{i \in I^-}[\Rr_+(-\w_i) + \rho_{i}(C)],
\end{equation*}
which can be rewritten as
 \begin{equation} \label{eq:rephrase-BV-statement}
    [C] \equiv \sum_{i \in I^+ \cup I^-} s_i \left[\R_+ (s_i \w_i) + \rho_i(C)\right].
\end{equation}
Therefore, it suffices to show that every cone on the right-hand side of~\eqref{eq:rephrase-BV-statement} can be decomposed into a finite family of cones of form~\eqref{Form 1} or \eqref{Form 2}. 

To this end, we let  $i\in I^+\cup I^-$ and we consider the cone 
\[\cf_i := \R_+ (s_i \w_i) + \rho_i(C).\]
%that is, for $i$ such that $\w_i \in \R^n \setminus L$, or equivalently, for $1\leq i\leq n-1$ such that $\w_i \cdot \w_n \neq 0$.
According to \autoref{def:notations},
\[\rho_i(C) = \cf\left(\rho_{i}({\w_1}), \dots, \rho_i(\w_{i-1}), \, \rho_i(\w_{i+1}), \dots, \rho_i({\w_n})\right).\]
Since ${\w_n}\in L$, we have $\rho_i({\w_n}) = {\w_n}$. Thus,
\begin{equation} \label{eq:cone-in-decomp-1-after-BV}
    \cf_{i} =\cf\left(s_i \w_i, \, \rho_{i}({\w_1}), \dots, \rho_i(\w_{i-1}), \, \rho_i(\w_{i+1}), \dots, \rho_i({\w_{n-1}}), \, \w_n\right).
\end{equation} 
Next, we determine the projections that arise on the right-hand side of equation~\eqref{eq:cone-in-decomp-1-after-BV}, by distinguishing the cases $i=n-1$ and $i\neq n-1$.
In the former case where $i=n-1$, it follows from equation~\eqref{eq:wi} that $\rho_i({\w_k}) = {\ve{\ell}_k}$ for all $1\leq k\leq n-2$. Therefore, equation~\eqref{eq:cone-in-decomp-1-after-BV} simplifies to
\begin{align*}\label{eq:cone-in-decomp-1-case1}
\cf_{n-1} 
&=  \cf(s_{n-1}\w_{n-1}, \ve{\ell}_1, \ve{\ell}_2, \dots, \ve{\ell}_{n-2}, {\w_n}) \nonumber\\
&=  \cf(s_{n-1}\w_{n-1}, {\w_1} - \frac{{\w_1}\cdot {\w_n}}{\w_{n-1}\cdot {\w_n}}\w_{n-1}, \dots,  \w_{n-2} - \frac{\w_{n-2}\cdot {\w_n}}{\w_{n-1}\cdot {\w_n}}\w_{n-1}, {\w_n}).
\end{align*}
In the latter case where $i\leq n-2$ and $\w_i \cdot \w_n \neq 0$, we have by equation~\eqref{eq:wi} 
\[\w_{n-1} = \frac{\w_{n-1} \cdot {\w_n}}{\w_i \cdot {\w_n}}(\w_i-\ve{\ell}_i),\]
so
\begin{align} \label{eq:rho_w_n-1}
\rho_i(\w_{n-1}) &= - \left(\frac{\w_{n-1} \cdot {\w_n}}{\w_i \cdot {\w_n}}\right) \ve{\ell}_i
    = - \left(\frac{\w_{n-1} \cdot {\w_n}}{\w_i \cdot {\w_n}}\right) \left(\w_i - \frac{\w_i \cdot \w_n}{\w_{n-1} \cdot \w_n} \w_{n-1}\right) \nonumber\\
    &= \w_{n-1} - \left(\frac{\w_{n-1} \cdot {\w_n}}{\w_i \cdot {\w_n}}\right) \w_i.
\end{align}
For $1\leq k \leq n-2$ such that $k\neq i$, we have
   \[ \w_{k} = \ve{\ell}_k + \frac{\w_{k} \cdot {\w_n}}{\w_{n-1} \cdot {\w_n}}\w_{n-1} = \ve{\ell}_k + \frac{\w_{k} \cdot {\w_n}}{\w_{n-1} \cdot {\w_n}}\cdot\frac{\w_{n-1} \cdot {\w_n}}{\w_i \cdot {\w_n}}(\w_i-\ve{\ell}_i).\]
It follows from \eqref{eq:li} that 
 \begin{equation} \label{eq:rho_wk}
    \rho_i(\w_{k}) =\ve{\ell}_k - \frac{\w_{k} \cdot {\w_n}}{\w_i \cdot {\w_n}}\ve{\ell}_i = \w_{k} - \frac{\w_{k}\cdot {\w_n}}{\w_i \cdot {\w_n}}\w_i.  
 \end{equation}
Using \eqref{eq:rho_w_n-1} and \eqref{eq:rho_wk}, equation~\eqref{eq:cone-in-decomp-1-after-BV} simplifies to 
\begin{equation*}\label{eq:cone-in-decomp-1-case2}
\cf_{i} =  \cf\left(s_i \w_i, \left\{\w_{k} - \frac{\w_{k}\cdot {\w_n}}{\w_i \cdot {\w_n}}\w_i 
\st 1\leq k\leq n-1, k\neq i\right\}, {\w_n}\right).    
\end{equation*}
We observe that the generators of $\cf_i$ in the latter case agree with those of $\cf_{n-1}$ in the former case, by setting $i=n-1$.
%in either case, the cone $\cf_i$ is always generated by  $s_i \w_i, \w_n$ and $\w_{k} - \frac{\w_{k}\cdot {\w_n}}{\w_i \cdot {\w_n}}\w_i$ for $1\leq k\leq n-1$ such that $k\neq i$.

Therefore, we can rewrite the cone $\cf_i$ %on the right-hand side of~\eqref{eq:rephrase-BV-statement} 
for any $i \in I^+\cup I^-$ in the form of
\[\cf_i=\cf\left(\frac{\ve{u}_1}{\|\ve{u}_1\|}, \dots, \frac{\ve{u}_{n-2}}{\|\ve{u}_{n-2}\|}, \ve{u}_{n-1}, \ve{w}_n\right),\]
where
\begin{align} 
    \{\ve{u}_1, \dots, \ve{u}_{n-2} \} &= \left\{\w_{k} - \frac{\w_{k}\cdot {\w_n}}{\w_i \cdot {\w_n}}\w_i \st 1\leq k\leq n-1, k\neq i\right\} \label{eq:u-and-w}\\
    \ve{u}_{n-1} &= s_i \w_i \label{eq:u-and-w-last}
\end{align}

If the cone $\cf_i$ is already of the form~\eqref{Form 1} or \eqref{Form 2}, then no further decomposition is needed. Otherwise, we apply the inductive hypothesis on the $(n-1)$-dimensional cone 
\[\cf\left(\frac{\ve{u}_1}{\|\ve{u}_1\|}, \dots, \frac{\ve{u}_{n-2}}{\|\ve{u}_{n-2}\|}, \ve{u}_{n-1}\right)\]
to get its (signed) decomposition into cones $\mathfrak{k}_{1}, \dots, \mathfrak{k}_{M}$ that are of the form~\eqref{Form 1} or \eqref{Form 2}. That is, each cone $\mathfrak{k}_m := \cf(\v_1, \dots, \v_{n-1})$ from $\{\mathfrak{k}_{1}, \dots, \mathfrak{k}_{M}\}$ either \eqref{Form 1} contains lines, or \eqref{Form 2} satisfies
    \begin{enumerate}[\rm(i)]
        \item \label{prop:a} $\|\v_j\| = 1$ for $j=1, 2,\dots, n-1$
        \item \label{prop:b} $\v_{n-1} = \ve{u}_{n-1}$
        \item \label{prop:c} $\langle \v_1, \v_2,\dots, \v_{n-2} \rangle = \langle \ve{u}_1, \ve{u}_2,\dots, \ve{u}_{n-2} \rangle$
        \item \label{prop:d} $\v_j \cdot \v_{n-1} =0$ for $j = 1, 2,\dots, n-3$
        \item \label{prop:e} 
        $\x^T M_{n-1}(\mathfrak{k}_m)\, \x = \left\lVert\begin{bmatrix}\epsilon_1\v_1, \dots, \epsilon_{n-1}\v_{n-1}\end{bmatrix} \x\right\rVert^2$ holds for every $\x\in \Rr^{n-1}$, where $\epsilon_j = \pm 1$.
    \end{enumerate}
Let 
\[C_m = \mathfrak{k}_m + \cf(\w_n) = \cf(\v_1, \dots, \v_{n-1}, \w_n)\] 
be the $n$-dimensional cone obtained by appending $\w_n$ to the generators of $\mathfrak{k}_m$. Then, the cones $C_1, C_2, \dots, C_M$ give a decomposition of the cone $\cf_i$ on the right-hand side of~\eqref{eq:rephrase-BV-statement}. It is clear that if $\mathfrak{k}_m$ contain lines, then $C_m$ also contains lines, so it is of the form~\eqref{Form 1}.  It remains to show that $C_m=\cf(\v_1, \dots, \v_{n-1}, \w_n)$ is of the form~\eqref{Form 2}, given that $\mathfrak{k}_m=\cf(\v_1, \dots, \v_{n-1})$ satisfies properties~\eqref{prop:a}--\eqref{prop:e}.

We set $\v_n = \w_n$. Then, $C_m=\cf(\v_1, \dots, \v_{n-1}, \v_n)$ clearly satisfies property~\eqref{property:D1-norm-is-1} because of \eqref{prop:a}, and property~\eqref{property:D1-vn=wn}. Property~\eqref{property:D1-span} holds because 
\[
\langle \v_1, \dots,  \v_{n-2}, \v_{n-1} \rangle = 
\langle \ve{u}_1, \dots, \ve{u}_{n-2}, \ve{u}_{n-1} \rangle =
\langle \w_1, \dots, \w_{n-2}, \w_{n-1} \rangle,
\]
where the first equality follows from properties~\eqref{prop:b} and \eqref{prop:c}, and the second equality follows from equations~\eqref{eq:u-and-w} and \eqref{eq:u-and-w-last}.
For any $1\leq k \leq n-1$, $k\neq i$, we have 
\[
\left(\w_{k} - \frac{\w_{k}\cdot {\w_n}}{\w_i \cdot {\w_n}}\w_i \right)\cdot \w_n = \w_{k}\cdot {\w_n} - \frac{\w_{k}\cdot {\w_n}}{\w_i \cdot {\w_n}}\, \w_i \cdot {\w_n}=  0.
\]
Thus, $\w_n$ is orthogonal to all $\ve{u}_1, \dots, \ve{u}_{n-2}$ by \eqref{eq:u-and-w}. Property~\eqref{property:D1-orthogonal-to-vn} holds because $\ve{v}_n = \w_n$ is also orthogonal to all $\ve{v}_1, \dots, \ve{v}_{n-2}$ by property~\eqref{prop:c}. %Finally, we will show that property~\eqref{property:D1-assoc-matrix-pos-def} holds as well. 
Using property~\eqref{property:D1-orthogonal-to-vn}, the associated matrix of $C_m$ can be written in relation to the associated matrix of $\mathfrak{k}_m$, as

\[
M_{n}(C_m)= 
\begin{bmatrix} 
\block(4,4){M_{n-1}(\mathfrak{k}_m)} & 0\\ 
& & & & \vdots \\
& & & & 0 \\& & & &  -|\v_{n-1}\cdot \v_n|\\ 
0 & \dots & 0 & -|\v_{n-1}\cdot \v_n| & 1
\end{bmatrix}.
\]
Let vector $\ve{x}=\vect{x_1, \dots, x_n}\in\R^n$. Then,
\begin{equation*}\label{eq:pos-def-mn}
\x^T M_n(C_m) \x = \begin{bmatrix} x_1 \cdots x_{n-1} \end{bmatrix} M_{n-1}(\mathfrak{k}_m)\begin{bmatrix} x_1 \\ \vdots \\ x_{n-1} \end{bmatrix}  -2|\v_{n-1}\cdot \v_n|\, x_{n-1}x_n + x_n^2.
\end{equation*}
By property~\eqref{prop:e}, for the appropriate choice of $\epsilon_1, \dots, \epsilon_{n-1} \in \{\pm 1\}$,
%\[\begin{bmatrix} x_1 \cdots x_{n-1} \end{bmatrix} M_{n-1}(\mathfrak{k}_m)\begin{bmatrix} x_1\\ \vdots \\ x_{n-1} \end{bmatrix} = \]
the first term on the right-hand side of the above equation %\eqref{eq:pos-def-mn} 
is equal to 
\[\left\lVert \epsilon_1x_1\v_1+\dots+\epsilon_{n-1}x_{n-1}\v_{n-1}\right\rVert^2.\]
By setting
\begin{equation*}
\epsilon_n =
\begin{cases}
-1, & \text{if $\v_{n-1}\cdot \v_{n}$ and $\epsilon_{n-1}$ have the same sign}\\
1, & \text{otherwise},
\end{cases}
\end{equation*}
we have that $-|\v_{n-1}\cdot \v_n| = \epsilon_{n-1} \epsilon_n\,(\v_{n-1}\cdot \v_n)$, and hence
\begin{align*}
&\x^T M_n(C_m) \x \\
 = & \left\lVert \epsilon_1x_1\v_1+\dots+\epsilon_{n-1}x_{n-1}\v_{n-1}\right\rVert^2 +  2 \epsilon_{n-1} \epsilon_n \,(\v_{n-1}\cdot \v_n) x_{n-1}x_n + \epsilon_n^ 2 x_n^2 \\
= & \left\lVert\begin{bmatrix}\epsilon_1\v_1, \dots, \epsilon_{n-1}\v_{n-1}, \epsilon_n\v_n\end{bmatrix} \x\right\rVert^2.
\end{align*}
This implies that property~\eqref{property:D1-assoc-matrix-pos-def} holds as well. We showed that in Case 2, every cone obtained via B--V decomposition can be further decomposed into a finite family of cones, each of which either contains lines or has an associated matrix that is positive definite. This concludes the proof of the theorem.
\end{proof}

\begin{corollary} \label{cor:decomp-1}
%Let $\v_1, \dots, \v_n$ form a basis of $\Rr^n$ and let cone $C=\cf(\v_1, \dots, \v_n)$. Then, 
The decomposition in \autoref{thm:decomp-1} allows for the computation of the normalized solid angle of a simplicial cone $C \in \R^n$ via the power series~\eqref{eq:t-alpha-series}.
\end{corollary}
\begin{proof}
Suppose \autoref{thm:decomp-1} yields the decomposition $C_1, \dots, C_N$ of $C$ such that
$[C] = \sum_{i=1}^N s_i [C_i]$, where $s_i = \pm 1$.
Then 
\[\tilde{\Omega}_n(C) = \sum_{i=1}^N s_i \tilde{\Omega}_n(C_i).\]
For each $C_i$ in the decomposition, if it is of the form~\eqref{Form 1}, then we can apply \autoref{cor:project-out-linear-subspaces} to reduce it to a lower-dimensional problem; if it is of the form~\eqref{Form 2}, then \autoref{thm:ribando2006} applies. By induction on the dimension, each $\tilde{\Omega}_n(C_i)$ can be computed.
\end{proof}

We note that in order to compute the solid angle measure using \autoref{thm:decomp-1}, one must determine the cones that contain lines and take orthogonal projections to reduce dimension, which is not immediately obvious. %\tblue{Not obvious what is meant by ``not immediately obvious''.}
As such, we present another decomposition method in the following section. 

\section{Second decomposition method}
\label{sec:decomp-lower-dim-cones}
The decomposition \autoref{thm:decomp-2} %and \tblue{Corollary 3.4} 
in this section gives a more direct way to compute the normalized solid angle measure of a simplicial cone. %\tblue{more to corollary} 
The resulting cones are either lower-dimensional cones or full-dimensional simplicial cones. In the former case, the cones have normalized solid angle measure 0. In the latter, the generators of the cones are explicitly given. This allows us to determine the normalized solid angle measure of the original cone, according to \autoref{cor:decomp-2}.

\subsection{Brion--Vergne decomposition with respect to a line}
We will make use another variant of Brion--Vergne (B--V) decomposition (see \cite{brion1997residue}, \cite{baldoni2013intermediate}), which is with respect to a one-dimensional subspace. Each cone in the decomposition will contain a generator along the given one-dimensional subspace. This B--V decomposition is reformulated below.

\begin{theorem}[{Brion--Vergne decomposition with respect to a line} {\cite[Proposition 15a]{baldoni2013intermediate}, rephrased}] 
\label{thm:BV-wrt-line}
Let $\w_1, \dots, \w_n$ form a basis of $\Rr^n$ and let cone $C =\cf({\w_1}, \dots, {\w_n})$. 
%Let $\ve{r} = \sum_{i=1}^n \lambda_i\w_i$ be a non-zero vector in $\R^n$. 
Given a non-zero vector $\ve{r}=\sum_{i=1}^n \lambda_i\w_i \in \Rr^n$, let $\delta_i\in \{-1, 0, 1\}$ denote the sign of $\lambda_i$ for $1\leq i\leq n$.
Then, the following relation holds, modulo indicator functions of cones containing lines.
\begin{equation*} \label{eq:bv-wrt-line}
[C] \equiv \sum_{i : \delta_i \neq 0} s_i\left[\cf\left(\epsilon_{i,1}\w_1, \dots, \epsilon_{i,i-1}\w_{i-1}, \epsilon_{i,i+1}\w_{i+1}, \dots ,\epsilon_{i,n}\w_n, \delta_i\ve{r}\right)\right],
\end{equation*}
where $s_i, \epsilon_{i,k} \in \{\pm 1\}$ for $1\leq k \leq n, k\neq i$ are given by
\begin{align}
    s_i &= \begin{cases}
        (-1)^{\mathbf{card}\left(\{1 \leq j < i\,\st\, \delta_j = 1\}\right)}, & \text{if } \delta_i = 1\\
        (-1)^{\mathbf{card}\left(\{i < j \leq n\,\st\, \delta_j = -1\}\right)}, & \text{if } \delta_i = -1
        \end{cases} \label{eq:si}\\ 
    \epsilon_{i,k} &= \begin{cases}
        -1, & \text{if } \delta_i = \delta_k = 1 \text{ and } k<i,\text{ or if } \delta_i = \delta_k = -1 \text{ and } k>i\\
       1, & \text{otherwise}
        \end{cases}.  \label{eq:epsik}
\end{align}
\end{theorem}

Recall that if $C=\cf(\w_1, \dots, \w_n)$ is a simplicial cone in $\R^n$ generated by some basis $\w_1, \dots, \w_n\in \R^n$, then its dual cone \[C^\ast := \{\ve{y}\in \R^n \st \ve{y}\cdot\ve{x} \geq 0 \; \forall \ve{x}\in C\}\] is also simplicial. The dual cone $C^\ast=\cf(\w^\ast_1\dots, \w^\ast_n)$ is generated by some dual basis $\w^\ast_1, \dots, \w_n^\ast\in \R^n$ such that for every $1\leq i, j\leq n$ and $j\neq i$, 
\begin{equation} \label{eq:dual-cone-generators}
    \w^\ast_i \cdot \w_i > 0 \text{ and  } \w^\ast_i \cdot \w_j = 0.
\end{equation}
Furthermore, the dual of the dual is the cone itself, i.e., $\left(C^\ast\right)^\ast = C$. The dual of a cone containing lines is a lower dimensional cone (i.e., a cone whose affine dimension is less than $n$).

By first passing to the dual cone, then applying Brion--Vergne decomposition \autoref{thm:BV-wrt-line} with respect to a particular line, and finally taking the dual again, we obtain the following decomposition which is modulo lower dimensional cones. Furthermore, each cone in the decomposition shares at least two generators $\w_i$ and $\w_n$ of the original cone.

\begin{theorem}
\label{thm:BV-mod-lower-dim-cone}
Let $\w_1, \dots, \w_n$ form a basis of $\Rr^n$ and let cone $C =\cf\left({\w_1}, \dots, {\w_n}\right)$. 
Let $\delta_i \in \{-1, 0, 1\}$ denote the sign of $\w_i \cdot \w_n$ for $1 \leq i \leq n-1$ and set $\delta_n = 0$. Suppose that the $\delta_i$'s are not all zero.
Then, the following relation holds, modulo indicator functions of lower dimensional cones.
\begin{equation*}
[C] \equiv \sum_{i : \delta_i \neq 0} s_i\left[\cf_i\right],
\end{equation*}
where each cone $\cf_i$ is generated by $\u_{i,k}$ $(1\leq k\leq n)$ defined as
\begin{equation} \label{eq:uik}
    \u_{i,k} = 
    \begin{cases}
     \w_k, & \text{if }\delta_k = 0  \text{ or } k=i \\
     \epsilon_{i,k} \left(\w_k - \frac{\w_k\cdot \w_n}{\w_i\cdot \w_n} \w_i \right), & \text{otherwise}
    \end{cases}
\end{equation}
and the signs $s_i, \epsilon_{i,k} \in \{\pm 1\}$ are given by equations~\eqref{eq:si} and \eqref{eq:epsik}, respectively.
\end{theorem}
\begin{proof} %\tblue{what if r is zero?}
Suppose that the dual cone $C^\ast$ is generated by  $\w^\ast_1, \dots, \w_n^\ast\in \R^n$. Let 
\[
\ve{r}= \sum_{i=1}^{n-1} \left(\frac{\w_i \cdot \w_n}{\w_i\cdot \w^\ast_i}\right)\w^\ast_i.
\]
We note that the signs $\delta_i$ agree with those defined in \autoref{thm:BV-wrt-line}, since $\w_i\cdot \w^\ast_i>0$. 
We apply \autoref{thm:BV-wrt-line} to the dual cone $C^\ast = \cf\left(\w^\ast_1, \dots, \w^\ast_n\right)$, with respect to the non-zero vector $\ve{r}$, and obtain that
\begin{equation*} 
[C^\ast] \equiv \sum_{i : \delta_i \neq 0} s_i\left[\cf\left(\epsilon_{i,1}\w^\ast_1, \dots, \epsilon_{i,i-1}\w^\ast_{i-1}, \epsilon_{i,i+1}\w^\ast_{i+1}, \dots ,\epsilon_{i,n}\w^\ast_n, \delta_i\ve{r}\right)\right],
\end{equation*}
modulo indicator functions of cones containing lines.

Since linear identities that hold for indicator functions of cones also hold for their
duals, we have
\begin{equation*} 
[C] \equiv \sum_{i : \delta_i \neq 0} s_i\left[\cf\left(\epsilon_{i,1}\w^\ast_1, \dots, \epsilon_{i,i-1}\w^\ast_{i-1}, \epsilon_{i,i+1}\w^\ast_{i+1}, \dots ,\epsilon_{i,n}\w^\ast_n, \delta_i\ve{r}\right)^\ast \right],
\end{equation*}
modulo indicator functions of lower dimensional cones. % (with affine dimension less than $n$).

We now show that each dual cone %$\cf\left(\epsilon_{i,1}\w^\ast_1, \dots, \epsilon_{i,i-1}\w^\ast_{i-1}, \epsilon_{i,i+1}\w^\ast_{i+1}, \dots ,\epsilon_{i,n}\w^\ast_n, \delta_i\ve{r}\right)^\ast$ 
on the right-hand side of the above relation is generated by the $\u_{i,k}$'s that are defined in \eqref{eq:uik}. To this end, we verify that
\begin{equation}\label{eq:check-dual-u-1}
    \w_i \cdot \left(\delta_i\ve{r}\right) >0 \text{ and } \w_i \cdot \left(\epsilon_{i,j}\w_j^\ast\right) = 0 \; \forall j\neq i, 
\end{equation}
and that for any $k\neq i$,
\begin{equation}\label{eq:check-dual-u-2}
    \u_{i,k}\cdot\left(\delta_i\ve{r}\right)=0, \; 
    \u_{i,k}\cdot\left(\epsilon_{i,k}\w^\ast_k\right)>0 \text{ and }
    \u_{i,k}\cdot\left(\epsilon_{i,j}\w^\ast_j\right)=0 \; \forall j\neq i \text{ or }k.
\end{equation}
The conditions in \eqref{eq:check-dual-u-1} trivially hold because of \eqref{eq:dual-cone-generators}. 
To show \eqref{eq:check-dual-u-2} for $k\neq i$, we first consider the case where $\delta_k =0$, which implies that $\w_k \cdot \w_n=0$ and $\u_{i,k}=\w_k$ by \eqref{eq:uik}. We have 
\[\w_{k}\cdot\left(\epsilon_{i,k}\w^\ast_k\right)>0, \quad  \w_{k}\cdot\left(\epsilon_{i,j}\w^\ast_j\right)=0 \; \forall j\neq i \text{ or }k, \text{ and }\]
\begin{equation*}
   \w_{k}\cdot\left(\delta_i\ve{r}\right)= \w_{k}\cdot\left(\delta_i \sum_{j=1}^{n-1} \left(\frac{\w_j \cdot \w_n}{\w_j\cdot \w^\ast_j}\right) \w^\ast_j  \right) 
\end{equation*}
which is equal to $0$ when $k=n$, or equal to $\delta_i \left(\w_k \cdot \w_n\right)=0$ when $k\neq n$.
Now consider $k\neq i$ such that $\delta_k \neq 0$, which implies $k\neq n$. Since $\delta_i\neq 0$, we also know that $i\neq n$ and $\w_i \cdot \w_n\neq 0$.
Using \eqref{eq:uik}, we obtain that
\begin{align*}
\u_{i,k} \cdot (\delta_i \ve{r}) &= \epsilon_{i,k} \, \delta_i \left(\w_k - \frac{\w_k\cdot \w_n}{\w_i\cdot \w_n} \w_i \right)\cdot \sum_{j=1}^{n-1} \left(\frac{\w_j \cdot \w_n}{\w_j\cdot \w^\ast_j}\right) \w^\ast_j \\
%&=\epsilon_{i,k} \, \delta_i \left[\left(\frac{\w_k \cdot \w_n}{\w_k\cdot \w^\ast_k}\right) \left(\w_k \cdot \w^\ast_k\right) -\left( \frac{\w_k\cdot \w_n}{\w_i\cdot \w_n}\right) \left(\frac{\w_i \cdot \w_n}{\w_i\cdot \w^\ast_i}\right)\left(\w_i \cdot \w^\ast_i\right) \right] \\
&=\epsilon_{i,k} \, \delta_i \left[\frac{\w_k \cdot \w_n}{\w_k\cdot \w^\ast_k} \left(\w_k \cdot \w^\ast_k\right) - \frac{\w_k\cdot \w_n}{\w_i\cdot \w_n} \cdot \frac{\w_i \cdot \w_n}{\w_i\cdot \w^\ast_i}\left(\w_i \cdot \w^\ast_i\right) \right] = 0\\
\u_{i,k}\cdot\left(\epsilon_{i,k}\w^\ast_k\right) &= \epsilon_{i,k}^2 \left(\w_k - \frac{\w_k\cdot \w_n}{\w_i\cdot \w_n} \w_i \right) \cdot \w^\ast_k = \w_k \cdot \w^\ast_k > 0 \\
\u_{i,k} \cdot \left(\epsilon_{i,j}\w^\ast_j\right)&= \epsilon_{i,k}\, \epsilon_{i,j} \left(\w_k - \frac{\w_k\cdot \w_n}{\w_i\cdot \w_n} \w_i \right) \cdot \w^\ast_j = 0 \quad \quad \forall j\neq i\text{ or }k.
\end{align*}
Therefore, the conditions in \eqref{eq:check-dual-u-2} are all satisfied. The theorem holds.
\end{proof}

\subsection{Solid angle decomposition that includes lower dimensional cones} %into solid angles lying in Domain of Convergence}

We present another decomposition theorem for a given full-dimensional simplicial
cone whose solid angle measure is of interest. The theorem is analogue to \autoref{thm:decomp-1}. The resulting cones from \autoref{thm:decomp-2} are either lower dimensional and hence have $0$ as normalized solid angle measures, or have positive definite associated matrices so that \autoref{thm:ribando2006} applies.

\begin{theorem} \label{thm:decomp-2}
Given linearly independent unit vectors $\w_1, \w_2,\dots, \w_n \in \Rr^n$, the cone $\cf(\w_1, \w_2, \dots, \w_n)$ can be decomposed into a finite family of cones, each of which is either \begin{enumerate} [\rm(I)]
    \item \label{item-I} a cone of affine dimension less than $n$, or
    \item \label{item-II} a full-dimensional cone generated by some vectors $\v_1, \dots, \v_n$ that satisfy the properties~\eqref{property:D1-norm-is-1}--\eqref{property:D1-assoc-matrix-pos-def} in \autoref{thm:decomp-1} so that its associated matrix is positive definite.
\end{enumerate}
\end{theorem}

\begin{proof}
    The proof is similar to that of \autoref{thm:decomp-1}.   
    We proceed by induction on the dimension. The base case $n=2$ is trivial (same proof as in \autoref{thm:decomp-1}). Suppose that the theorem holds for dimension $n-1$. Now consider the cone $\cf(\w_1, \w_2, \dots, \w_n)$ in dimension $n$ which is not already of the form~\eqref{Form 1} or \eqref{Form 2}.
    
    If $\w_n$ is orthogonal to all $\w_i$ for $1\leq i\leq n-1$, then the arguments in \textbf{Case 1} of the proof of \autoref{thm:decomp-1} (where we replace ``contain lines'' by ``of affine dimension less than $n$'') verbatim apply, so the statement is true for dimension $n$. 
    
    Next, we assume that $\w_n$ is not orthogonal to all $\w_1,\dots, \w_{n-1}$. Let $\delta_n =0$ and $\delta_i \in \{-1, 0, 1\}$ be the sign of $\w_i\cdot\w_n$ for $1\leq i\leq n$. Then, the $\delta_i$'s are not all zero. We apply \autoref{thm:BV-mod-lower-dim-cone}, and obtain the (signed) decomposition
\begin{equation}\label{eq:BV-in-decomp-2}
[C] \equiv \sum_{i : \delta_i \neq 0} s_i\left[\cf_i\right] \text{ modulo indicator functions of lower dimensional cones,}
\end{equation}
where each cone $\cf_i = \cf\left(\u_{i,1}, \u_{i,2}, \dots,\u_{i,n}\right)$ is generated by the $\u_{i,k}$'s according to equation~\eqref{eq:uik}. In particular, $\u_{i,i} = \w_i$ and $\u_{i,n}=\w_n$.
%and the signs $s_i, \epsilon_{i,k} \in \{\pm 1\}$ are given by equations~\eqref{eq:si} and \eqref{eq:epsik}, respectively. 
Thus, we can rewrite the cone $\cf_i$ for any $1\leq i \leq n-1$ such that $\delta_i\neq 0$ in the form of 
\[\cf_i = \cf\left(\u_1,\dots, \u_{n-2}, \u_{n-1}, \w_n \right)\] with $\{\ve{u}_1, \dots, \ve{u}_{n-2} \} = \left\{\frac{\u_{i,k}}{\|\u_{i,k}\|} \st 1\leq k\leq n-1, k\neq i\right\}$ and $\ve{u}_{n-1} = \w_i$.
%\begin{align*} 
%    \{\ve{u}_1, \dots, \ve{u}_{n-2} \} &= \left\{\frac{\u_{i,k}}{\|\u_{i,k}\|} \st 1\leq k\leq n-1, k\neq i\right\}\\
%    \ve{u}_{n-1} &= \w_i.
%\end{align*}

If the cone $\cf_i$ is already of the form~\eqref{Form 1} or \eqref{Form 2}, then no further decomposition is needed. Otherwise, by the inductive hypothesis, the $(n-1)$-dimensional cone 
$\cf\left(\ve{u}_1, \dots, \ve{u}_{n-2}, \ve{u}_{n-1}\right)$
can be decomposed into cones $\mathfrak{k}_{1}, \dots, \mathfrak{k}_{M}$,
such that each cone $\mathfrak{k}_{m} := \cf(\v_1, \dots, \v_{n-1})$ from $\{\mathfrak{k}_{1}, \dots, \mathfrak{k}_{M}\}$ is either \eqref{item-I} a cone of affine dimension less than $n-1$, or \eqref{item-II} a cone with positive associated matrix satisfying the properties~\eqref{prop:a}--\eqref{prop:e}.
Let 
\[C_m = \mathfrak{k}_m + \cf(\w_n) = \cf(\v_1, \dots, \v_{n-1}, \w_n)\] 
be the $n$-dimensional cone obtained by appending $\w_n$ to the generators of $\mathfrak{k}_m$. 
Then, the same arguments as in the proof of \autoref{thm:decomp-1} show that $C_m$ has the desired properties, and that the cones $C_1, \dots, C_M$ give a (signed) decomposition of the cone $\cf_i$ in the relation~\eqref{eq:BV-in-decomp-2}. This concludes the inductive step, and therefore, the theorem holds.
\end{proof}

\begin{corollary} \label{cor:decomp-2}
Let $C$ be a simplicial cone in $\R^n$. \autoref{thm:decomp-2} gives explicitly the cones $C_i$ whose normalized solid angle measures $\tilde{\Omega}_n(C_i)$ can be computed via the power series formula~\eqref{eq:t-alpha-series} and the signs $s_i \in \{\pm 1\}$, such that the normalized solid angle measure of $C$ satisfies
\begin{equation}\label{eq:solid-angle-decomp-sum}
  \tilde{\Omega}_n(C) = \sum_{i=1}^N s_i \tilde{\Omega}_n(C_i).  
\end{equation}
Furthermore, the number of cones in the decomposition is $N\leq (n-1)!$.
\end{corollary}
\begin{proof}
 Suppose that \autoref{thm:decomp-2} yields a signed decomposition of the cone $C$: %into $N \leq (n-1)!$ cones $C_1, \dots, C_N$:
\[
\left[C\right] \equiv \sum_{i=1}^N s_i \left[C_i\right] \text{ modulo indicator functions of lower dimensional cones.}
\]

Since the normalized solid angle measure of a lower dimensional cone with respect to $\R^n$ is zero, equation~\eqref{eq:solid-angle-decomp-sum} holds.

Since each full-dimensional cone $C_i$ resulting from the decomposition has a positive-definite matrix, \autoref{thm:ribando2006} applies, and so $\tilde{\Omega}_n(C_i)$ can be computed via the power series~\eqref{eq:t-alpha-series}.

Finally, we show that $N\leq (n-1)!$ by induction on $n$. If $n=2$, then the base case in the proof of \autoref{thm:decomp-1} shows that the cone $C$ is already of the form~\eqref{item-II}, so no further decomposition is needed. Suppose that an $(n-1)$-dimensional
cone can be decomposed into at most $(n-2)!$ cones. For an $n$-dimensional cone $C$ whose associated matrix is not yet positive-definite, \autoref{thm:decomp-2} first decomposes it into the cones $\cf_i$ for $i$ such that $\delta_i\neq 0$, according to \eqref{eq:BV-in-decomp-2}. Since $\delta_n=0$, there are at most $n-1$ such cones $\cf_i =\cf(\u_{i,1}, \dots, \u_{i,n})$. Subsequently, for  each $\cf_i$ whose associated matrix is not yet positive-definite, we omit its generator $\u_{i,n}=\w_n$ in order to obtain the $(n-1)$-dimensional cone $\cf(\u_{i,1}, \dots, \u_{i,n-1})$, and we further decompose it into at most $(n-2)!$ cones that satisfy the desired properties by the inductive hypothesis. Therefore, the total number of cones resulting from \autoref{thm:decomp-2} is at most $(n-1)(n-2)!=(n-1)!$.
\end{proof}

\section{Asymptotic bound on the truncation error}
\subsection{Tridiagonal associated matrices}
As aforementioned, the large number of coordinates needed creates issues with computational feasibility of the power series formula~\eqref{eq:t-alpha-series} in \autoref{thm:ribando2006}. Ribando \cite[p.~487]{ribando2006measuring} states that ``\textit{\elide  accurate series approximations will require theorems allowing us to reduce the number of terms that need computing.}'' 
Note that when $\alpha_{ij}=0$, the only terms of $\ve{\alpha}^{\ve{a}}$ in the series~\eqref{eq:t-alpha-series} that are non-zero must have multiexponent $\ve{a}$ whose $a_{ij}=0$. 
Thus, one possibility for reducing the number of terms needed for computing is by decomposing into cones with as many pairwise orthogonal generators as possible. 
One particular interesting case is when 
$\left(\alpha_{ij}\right)_{1\leq i,j\leq n}$ is a tridiagonal matrix (i.e., $a_{ij}=\v_i\cdot\v_j=0$ whenever $i\neq j$ and $i+1 \neq j$), as we will discuss in this section.

Given linearly independent unit vectors $\v_1, \v_2,\dots, \v_n \in \Rr^n$, let $V\in \R^{n\times n}$ denote the matrix whose $i$-th column vector is $\v_i$.
Suppose that $V^T V$ is the following symmetric and  tridiagonal matrix, where $\beta_i = \v_i \cdot \v_{i+1}$ for $1\leq i\leq n-1$.
\begin{equation}\label{eq:VTV}
V^TV = \begin{bmatrix}
1 & \beta_1 & 0 & \dots & 0\\
\beta_1 & 1 & \beta_2 & \ddots & \vdots\\
0 &  \ddots & \ddots & \ddots & 0\\
\vdots  &   & \beta_{n-2} & 1 & \beta_{n-1}\\
0 &  \dots & 0 & \beta_{n-1} & 1
\end{bmatrix}.
\end{equation}
Then, the formula~\eqref{eq:t-alpha-series} simplifies to the following multivariate power series $T_{\ve \beta}$ in $(n-1)$-variables.
\begin{equation}\label{eq:t-beta-series}
\begin{split}
T_{\ve \beta} = \frac{\mathopen| \det V|}{(4\pi)^{\frac{n}{2}}} \sum_{{\ve b}\in \N^{n-1}} \left[ \frac{(-2)^{\sum b_{i}}}{\prod_{i=1}^{n-1} b_{i}!}\,
\Gamma\left(\frac{1+ b_1}{2}\right) 
 \Gamma\left(\frac{1+ b_1+b_2}{2}\right) 
 \cdots \right.\\
 \left. \Gamma\left(\frac{1+ b_{n-2}+b_{n-1}}{2}\right) \Gamma\left(\frac{1+b_{n-1}}{2}\right) \right] \ve\beta^{\ve{b}} .
 %\beta_1^{b_1}\beta_2^{b_2}\cdots\beta_{n-1}^{b_{n-1}}.
 \end{split}
\end{equation}

\begin{theorem}\label{thm:tridiag-matrix}
Let $C = \cf(\v_1, \v_2, \dots, \v_n)$ be a cone in $\R^n$ generated by the linearly independent unit vectors $\v_1, \v_2, \dots, \v_n$. Let $V$ be the matrix whose column vectors are $\v_1, \v_2, \dots, \v_n$. Suppose that $V^T V$ is a tridiagonal matrix as in~\eqref{eq:VTV}. Then, the associated matrix $M_n(C)$ has the same eigenvalues as $V^T V$. 

In particular, if $V^T V$ is tridiagonal, then $M_n(C)$ is positive definite, and the power series~\eqref{eq:t-beta-series} converges absolutely to the normalized solid angle measure $\tilde\Omega_n(C)$ of the cone $C$.
\end{theorem}
\begin{proof} When $V^T V$ is the symmetric and tridiagonal matrix given by equation~\eqref{eq:VTV}, the associated matrix of $C$ is
\begin{equation} \label{eq:mnc-tridiag}
M_n(C) = \begin{bmatrix}
1 & -|\beta_1| & 0 & \dots & 0\\
-|\beta_1| & 1 & -|\beta_2| & \ddots & \vdots\\
0 &  \ddots & \ddots & \ddots & 0\\
\vdots  &   & -|\beta_{n-2}| & 1 & -|\beta_{n-1}|\\
0 &  \dots & 0 & -|\beta_{n-1}| & 1
\end{bmatrix},
\end{equation}
which is also symmetric and tridiagonal.

Let $P_j(\lambda)$ be the characteristic polynomial of the $j$-th leading principal minor of $V^T V$ for $1\leq j\leq n$. Set $\beta_0=0$. We have $P_0(\lambda)=1,\; P_1(\lambda)=1-\lambda$ and
\[P_j(\lambda) = (1 - \lambda)P_{j-1}(\lambda) - \beta_{j-1}^2 P_{j-2}(\lambda) \;\text{ for } 2\leq j \leq n,\] 
where the last recurrence relation is a well-known result for symmetric tridiagonal matrix (e.g., see~\cite{parlett1998symmetric}). 
 Since relation depends only on $\beta_{j-1}^2 = \left(- \left|\beta_{j-1}\right|\right)^2$, it is clear that the characteristic polynomial of the associated matrix is the same as that of $V^TV$. Therefore, $V^T V$ and $M_n(C)$ have the same eigenvalues.
 
 Since the columns of $V$ are linearly independent, for any $\ve{x}\neq\ve{0}$, we have $\x^T(V^T V)\x=\|V\x\|^2> 0$, showing that $V^T V$ is positive definite. Thus, $M_n(C)$ is also positive definite. It follows from \autoref{thm:ribando2006} that the power series~\eqref{eq:t-beta-series} converges absolutely to the normalized solid angle measure $\tilde\Omega_n(C)$ of the cone $C$.
\end{proof}

\begin{remark}
We can strengthen property~\eqref{property:D1-orthogonal-to-vn} from ``$\v_i \cdot \v_n =0$ for $i = 1,\dots, n-2$''
    in \autoref{thm:decomp-1} and \autoref{thm:decomp-2} to
\begin{equation}\label{prop:IId-new}\tag{IId'}
    \v_i \cdot \v_j =0 \text{ for all }1 \leq i, j \leq n \text{ such that } j\neq i \text{ or } i\pm 1.
\end{equation}
The proofs of the two theorems hold verbatim. The new property~\eqref{prop:IId-new} ensures that $V^T V$ is a tridiagonal matrix for any cone from the decomposition, so that \autoref{thm:tridiag-matrix} applies. Therefore, we have the following variant of \autoref{cor:decomp-2}, which takes advantage of the tridiagonal structure. We note that the number of cones resulting from the new decomposition can be larger than that of \autoref{cor:decomp-2}, but it is still upper bounded by $(n-1)!$.
\end{remark}

\begin{corollary}\label{cor:tridiag-decomp}
Let $C$ be a simplicial cone in $\R^n$. \autoref{thm:decomp-2} (with property~\eqref{property:D1-orthogonal-to-vn} replaced by \eqref{prop:IId-new}) gives explicitly the cones $C_i$ and the signs $s_i \in \{\pm 1\}$, such that the normalized solid angle measure of $C$ satisfies
\begin{equation*}
  \tilde{\Omega}_n(C) = \sum_{i=1}^N s_i \tilde{\Omega}_n(C_i).  
\end{equation*}
The cones $C_i$ have positive-definite and tridiagonal associated matrices, so their normalized solid angle measures $\tilde{\Omega}_n(C_i)$ can be computed via the simplified power series formula~\eqref{eq:t-beta-series}.
Furthermore, the number of cones in the decomposition is $N\leq (n-1)!$.
\end{corollary}

\subsection{Eigenvalues and series truncation errors}

In their study \cite{gourion2010deterministic}, Gourion and Seeger observed that a particular power series exhibited slow convergence when the associated matrix was close to being singular. Considering this observation alongside the previous \autoref{thm:tridiag-matrix}, it is reasonable to investigate the impact of eigenvalues on the convergence of the power series. 

We notice that if the linearly independent unit vectors $\ve{v}_i$'s are all pairwise orthogonal, then the normalized solid angle measure of the cone $C=\cf(\v_1, \hdots, \v_n)$ is  trivial:
\[\tilde{\Omega}_n(C)= \frac{1}{2^{n}}.\] 
Therefore, in the following, we further assume that $n\geq 2$ and that $\v_i \cdot \v_j\neq 0$ for some $1\leq i < j \leq n$. Let $V$ denote the matrix whose $i^{th}$ column is $\v_i$.  Let $\lambda_{\min}$ be the smallest eigenvalue of $V^T V$.

\begin{lemma}\label{lemma:lambdamin}
    The smallest eigenvalue $\lambda_{\min}$ of $V^T V$ satisfies $0< \lambda_{\min} <1$.
\end{lemma}
\begin{proof}
The matrix $V^T V$ is positive definite, since for any $\ve{x}\neq\ve{0}$, we have that $\x^T(V^T V)\x=\|V\x\|^2> 0$, where the last strict inequality follows from the linear independence of the columns of $V$. Thus, $\lambda_{\min}>0$. 

By Cauchy interlacing theorem, $\lambda_{\min}$ is less than or equal to the minimum eigenvalue of a principal submatrix of $V^T V$. In particular, by taking the principal submatrix 
$\begin{bmatrix} 1 &\v_i \cdot \v_j\\ \v_i \cdot \v_j  & 1\end{bmatrix}$ where $\v_i \cdot \v_j\neq 0$, we have $\lambda_{\min}\leq 1 - \v_i \cdot \v_j < 1$.
\end{proof}

For the rest of this section, we focus on the case where $V^T V$ is the tridiagonal matrix as defined in equation~\eqref{eq:VTV}.
%\begin{assumption}
%Assume that $V^T V$ is the tridiagonal matrix as defined in equation~\eqref{eq:VTV}. $n\geq 2$. %$\ve\beta\neq \ve{0}$.
%\end{assumption}
%When $V^T V$ is a tridiagonal matrix, 
\autoref{thm:tridiag-matrix} implies that the power series $T_{\ve \beta}$ in $n-1$ variables as defined in equation~\eqref{eq:t-beta-series} converges absolutely to the solid angle measure. For simplicity, we consider only the series part 
\begin{equation*}
\frac{(4\pi)^{n/2}}{|\det V|} T_{\ve\beta} =  \sum_{{\ve b}\in \N^{n-1}} A_{\ve b}\, \ve\beta^{\ve{b}},
\end{equation*}  
where, for any multiexponent $\ve{b} = (b_1, \hdots, b_{n-1})$ in $\mathbb{N}^{n-1}$,
\begin{equation}\label{eq:Ab}
\begin{split}
 A_{\ve b} := \frac{(-2)^{\sum b_{i}}}{\prod_{i=1}^{n-1} b_{i}!}\,
\Gamma\left(\frac{1+ b_1}{2}\right) 
 &\Gamma\left(\frac{1+ b_1+b_2}{2}\right) 
 \cdots\\
 &\Gamma\left(\frac{1+ b_{n-2}+b_{n-1}}{2}\right) \Gamma\left(\frac{1+b_{n-1}}{2}\right).
\end{split}
\end{equation} 
We regard 
\begin{equation*}
    T(\x) = \sum_{{\ve b}\in \N^{n-1}} A_{\ve b}\, x_1^{b_1} x_2^{b_2}\cdots x_{n-1}^{b_{n-1}}
\end{equation*}
as a hypergeometric series of $(n-1)$ variables $\x=(x_1, \dots, x_{n-1})$ in Horn's sense \cite{horn1889convergenz}. 

\begin{remark}\label{rk:boundary-cvg}
We recall some standard notations and results from \cite{horn1889convergenz, ribando2006measuring} regarding the convergence of a hypergeometric series. We rephrase them below for the triadiagonal case  to describe the domain of convergence of $T(\x)$. 

Let $\ve{e}_i$ denote the $i$-th standard basis vector. 
For $1\leq i \leq n-1$, we define the ratio of the neighboring coefficients
 \begin{equation}\label{eq:fi}
     f_i(\ve{b}) :=\frac{A_{\ve{b}+\ve{e}_i}}{A_{\ve{b}}},
 \end{equation} 
and introduce the limit
\begin{equation*}
   \Psi_i(\ve{b}) := \lim_{t \rightarrow \infty} f_i(t\ve{b}).
\end{equation*}
We can view $A_{\ve{b}}, f_i(\ve{b})$ and $\Psi_i(\ve{b})$ as functions of $\ve{b}\in \R_{+}^{n-1}$ instead of $\ve{b}\in \N^{n-1}$. Then, for $1\leq i \leq n-1$, $f_i$ is a rational function of the variables $\ve{b}\in \R_{+}^{n-1}$ . Set $b_0=b_n =0$. The function $\Psi_i$ satisfies
\begin{equation}\label{eq:psi}
   \Psi_i(\ve{b}) = \lim_{t \rightarrow \infty} \frac{A_{t\ve{b}+\ve{e}_i}}{A_{t\ve{b}}}=  - \frac{\sqrt{(b_{i-1}+b_i)(b_i+b_{i+1})}}{b_i},
\end{equation}
and it is a rational and homogeneous function of degree zero.

The parameterized curve $\left(\frac{1}{|\Psi_1(\ve{b})|},\dots, \frac{1}{|\Psi_{n-1}(\ve{b})|}\right)$ %for $\ve{b}\in \R_{+}^{n-1}$ 
describes an hypersurface that bounds the convergence domain of the hypergeometric series $T(\x)$. That is, if a point $\ve{x}$ lies on the boundary of the convergence domain, then for some $\ve{b}\in \R_{+}^{n-1}$, 
\begin{equation}\label{eq:x-on-boundary}
    |x_i| = \frac{1}{|\Psi_i(\ve{b})|} \quad\quad \forall\; 1\leq i\leq n-1.
\end{equation}
In addition, 
\cite[Theorem 3.2]{ribando2006measuring} states that $\ve{x}$ lies on the boundary of the convergence domain, if
\begin{equation}\label{eq:boundary-det-0}
  \det \begin{bmatrix}
1 & -|x_1| & 0 & \dots & 0\\
-|x_1| & 1 & -|x_2| & \ddots & \vdots\\
0 &  \ddots & \ddots & \ddots & 0\\
\vdots  &   & -|x_{n-2}| & 1 & -|x_{n-1}|\\
0 &  \dots & 0 & -|x_{n-1}| & 1
\end{bmatrix} = 0.
\end{equation}
\end{remark}

It follows from \autoref{thm:tridiag-matrix} that in the tridiagonal case, the minimum eigenvalue of  the associated matrix $M(C)$ given in~\eqref{eq:mnc-tridiag} is also $\lambda_{\min}$, where $0<\lambda_{\min}<1$ by \autoref{lemma:lambdamin}, and that $\x=\ve{\beta}$ lies in the convergence domain of $T(\x)$.
The following lemma suggests that $1-\lambda_{\min}$ plays a role in the convergence of $T(\x)$.
\begin{lemma}\label{lemma:beta-lambda-on-boundary}
    The point $\left(\frac{\beta_1}{1-\lambda_{\min}}, \dots,\frac{\beta_{n-1}}{1-\lambda_{\min}}\right)$ lies on the boundary of the domain of convergence of the series $T(\ve{x})$.
\end{lemma}
\begin{proof}
    By \autoref{rk:boundary-cvg}, it suffices to show that equation~\eqref{eq:boundary-det-0} holds for $\ve{x}=\frac{\ve{\beta}}{1-\lambda_{\min}}$, which is clearly satisfied since $\lambda_{\min}$ is an eigenvalue of $M(C)$.
\end{proof}
We show a lemma regarding the monotonicity of the functions $f_i$ in \eqref{eq:fi}, which will become useful later in analysing the series truncation errors.
\begin{lemma}\label{lemma:fi-monotone}
    Let $\ve{b}\in \R_+^{n-1}$ and let $1\leq i, j\leq n-1$ such that $i\neq j$. Then, $|f_i(\ve{b}+\ve{e}_j)| \geq |f_i(\ve{b})|$.
\end{lemma}
\begin{proof}
    If $j\neq i\pm 1$, then it follows from equations~\eqref{eq:fi} and \eqref{eq:Ab} that $|f_i(\ve{b}+\ve{e}_j)| = |f_i(\ve{b})|$. If $j=i \pm 1$, then the desired inequality is equivalent to
    \begin{equation*}
        \frac{\Gamma\left(\frac{3+b_i+b_j}{2}\right)}{\Gamma\left(\frac{2+b_i+b_j}{2}\right)} \geq
        \frac{\Gamma\left(\frac{2+b_i+b_j}{2}\right)}{\Gamma\left(\frac{1+b_i+b_j}{2}\right)},
    \end{equation*}
    % \(\Gamma\left(\frac{3+b_i+b_j}{2}\right)\) \Big/ \(\Gamma\left(\frac{2+b_i+b_j}{2}\right)\)
    which holds because the $\Gamma$ function is log-convex on the positive real axis.
\end{proof}
We are interested in an asymptotic analysis on the truncation error of the series $T(\ve{\beta})$.  Truncating the series in partial degrees $\left(N_1,\dots, N_{n-1}\right)$, the error term is bounded by $E(N_1, \hdots, N_{n-1})$ defined below.
 
\begin{definition}\label{def:error-series}
For partial degrees $\left(N_1,\dots, N_{n-1}\right) \in \mathbb{N}^{n-1}$, define the series $T(\ve{\beta})$ truncation error as
 \[E(N_1, \dots, N_{n-1})  = \sum_{\ve{b}\in B} \left|A_{\ve b} \ve{\beta}^{\ve{b}}\right|,\]
 where $B = \{\ve{b} \in \mathbb{N}^{n-1} \st b_i \geq N_i \text{ for at least one } i\}$.
\end{definition}

We investigate the asymptotic decay of $E(N_1,\dots,N_{n-1})$, in relation to $1-\lambda_{\min}$.

\begin{theorem} \label{thm:error-any-dim}
For any $\rho$ such that $1-\lambda_{\min} < \rho < 1$, there exist partial degrees $N_1, \hdots, N_{n-1}$ such that for any integer $\ell \geq 1$, we have 
\begin{equation*}
    E(N_1+\ell, \dots, N_{n-1}+\ell) \leq \rho^\ell\,E(N_1, \dots, N_{n-1}).
\end{equation*}
\end{theorem}

To simplify the notation, we will prove \autoref{thm:error-any-dim} specifically for $n = 3$, which we restate as \autoref{thm:error-dim-3} below. Note that the proof for general $n$ follows in the exact same manner. 
%The proof of \autoref{thm:error-dim-3} borrows methods used by Horn in \cite{horn1889convergenz} regarding the convergence of the double hypergeometric series. (See \autoref{rk:boundary-cvg}.)

\begin{proposition} \label{thm:error-dim-3}
For any $\rho$ such that $1-\lambda_{\min} < \rho < 1$, there exist partial degrees $N_1$ and $N_2$ such that for any integer $\ell \geq 1$, we have 
\begin{equation*}
    E(N_1+\ell, N_2+\ell) \leq \rho^\ell \, E(N_1, N_2).
\end{equation*}
\end{proposition}

\begin{proof}
For the multiexponent $(b_1, b_2)$, the coefficient $A_{\ve{b}}$ in equation~\eqref{eq:Ab} writes
\begin{equation*}
   A_{b_1,b_2} = \frac{(-2)^{b_1+b_2}}{b_1!\, b_2!}\,\Gamma\left(\frac{1+b_1}{2}\right)\Gamma\left(\frac{1+b_2}{2}\right)\Gamma\left(\frac{1+b_1+b_2}{2}\right).  
\end{equation*}
By \autoref{lemma:beta-lambda-on-boundary} and \eqref{eq:x-on-boundary},  there exist $x_1, x_2 \in \Rr_{+}$ such that 
\begin{equation*}
    \frac{|\beta_1|}{1-\lambda_{\min}} = \frac{1}{|\Psi_1(x_1, x_2)|}\quad \text{ and }\quad  \frac{|\beta_2|}{1-\lambda_{\min}} = \frac{1}{|\Psi_2(x_1, x_2)|}
\end{equation*}
Thus,
\begin{equation}\label{eq:beta-lambda}
    \left|\beta_1 \Psi_1(x_1, x_2)\right| = 1-\lambda_{\min}\quad \text{ and }\quad \left|\beta_2 \Psi_2(x_1, x_2)\right| = 1-\lambda_{\min}
\end{equation}
Since $\Psi_1, \Psi_2$ are rational and homogeneous functions of degree zero by \eqref{eq:psi}, their values only depend on the ratio between $x_1$ and $x_2$. (We note that for general $n>3$, each $\Psi_i$ only depends on at most two consecutive pairwise ratios.) 
We express $(x_1, x_2)$ in polar coordinates $(r_0, \theta_0)$ with $r_0 = \sqrt{x_1^2+x_2^2}$ and $\theta_0 = \arctan(x_2 / x_1)$. It follows from \eqref{eq:psi} that $\left|\Psi_1(r \cos \theta, r\sin \theta)\right| = \sqrt{1+\tan\theta}$, which is increasing on $\theta \in [0,\pi/2)$, and $\left|\Psi_2(r \cos \theta, r\sin \theta)\right| = \sqrt{1+\cot\theta}$, which is decreasing on $\theta \in (0,\pi/2]$.
We obtain from \eqref{eq:beta-lambda} that 
\begin{equation}\label{eq:beta-lim-f}
    \left|\beta_1 \lim_{r\to \infty}f_1(r\cos\theta_0, r\sin\theta_0)\right| 
    = \left|\beta_2 \lim_{r\to \infty}f_2(r\cos\theta_0, r\sin\theta_0)\right|
    = 1-\lambda_{\min}.
\end{equation}
Let $\epsilon$ be a small positive number. Specifically, we set  
\begin{equation}\label{eq:eps-def}
  \epsilon :=   \sqrt{\frac{\rho}{1-\lambda_{\min}}} - 1 > 0. 
\end{equation}
It follows from \eqref{eq:beta-lim-f} that there exists $r_0 > 0$ such that
\begin{align*}
    \forall\; r \geq r_0, \quad \left|\beta_i f_i(r\cos\theta_0, r\sin\theta_0)\right| \leq (1-\lambda_{\min})(1+\epsilon/2) \; \text{ for } i =  1, 2.
\end{align*}

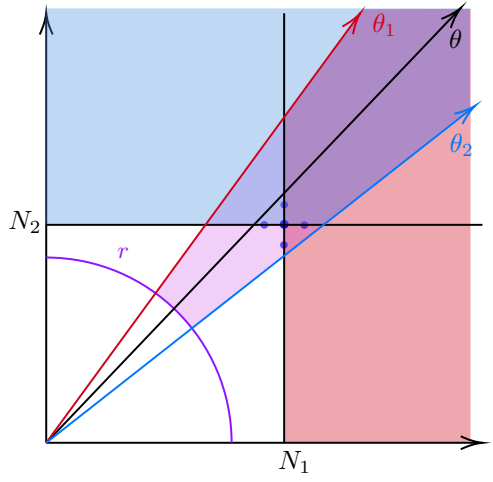
\begin{figure}[tp]
\centering

\tikzset{every picture/.style={line width=0.75pt}} %set default line width to 0.75pt        

\begin{tikzpicture}[x=0.75pt,y=0.75pt,yscale=-1,xscale=1]
%uncomment if require: \path (0,415); %set diagram left start at 0, and has height of 415

%Shape: Polygon [id:ds28354768508182904] 
\draw  [draw opacity=0][fill={rgb, 255:red, 208; green, 2; blue, 27 }  ,fill opacity=0.35 ] (189.83,1.17) -- (244,1) -- (244,139.2) -- (244,220) -- (150,220) -- (150.13,55.25) -- cycle ;
%Straight Lines [id:da08928903940554767] 
\draw    (30,220) -- (248,220) ;
\draw [shift={(250,220)}, rotate = 180] [color={rgb, 255:red, 0; green, 0; blue, 0 }  ][line width=0.75]    (10.93,-3.29) .. controls (6.95,-1.4) and (3.31,-0.3) .. (0,0) .. controls (3.31,0.3) and (6.95,1.4) .. (10.93,3.29)   ;
%Straight Lines [id:da5564893924469259] 
\draw    (30,220) -- (30,5) ;
\draw [shift={(30,3)}, rotate = 90] [color={rgb, 255:red, 0; green, 0; blue, 0 }  ][line width=0.75]    (10.93,-3.29) .. controls (6.95,-1.4) and (3.31,-0.3) .. (0,0) .. controls (3.31,0.3) and (6.95,1.4) .. (10.93,3.29)   ;
%Shape: Circle [id:dp7470555366901204] 
\draw  [color={rgb, 255:red, 0; green, 0; blue, 255 }  ,draw opacity=1 ][fill={rgb, 255:red, 0; green, 0; blue, 255 }  ,fill opacity=1 ] (148.25,109.75) .. controls (148.25,108.78) and (149.03,108) .. (150,108) .. controls (150.97,108) and (151.75,108.78) .. (151.75,109.75) .. controls (151.75,110.72) and (150.97,111.5) .. (150,111.5) .. controls (149.03,111.5) and (148.25,110.72) .. (148.25,109.75) -- cycle ;
%Shape: Circle [id:dp020845475742352626] 
\draw  [color={rgb, 255:red, 0; green, 0; blue, 255 }  ,draw opacity=0.67 ][fill={rgb, 255:red, 0; green, 0; blue, 255 }  ,fill opacity=0.67 ] (158.75,110.14) .. controls (158.75,109.35) and (159.39,108.71) .. (160.18,108.71) .. controls (160.97,108.71) and (161.61,109.35) .. (161.61,110.14) .. controls (161.61,110.93) and (160.97,111.57) .. (160.18,111.57) .. controls (159.39,111.57) and (158.75,110.93) .. (158.75,110.14) -- cycle ;
%Shape: Circle [id:dp6060112762161378] 
\draw  [color={rgb, 255:red, 0; green, 0; blue, 255 }  ,draw opacity=0.66 ][fill={rgb, 255:red, 0; green, 0; blue, 255 }  ,fill opacity=0.58 ] (138.57,110) .. controls (138.57,109.21) and (139.21,108.57) .. (140,108.57) .. controls (140.79,108.57) and (141.43,109.21) .. (141.43,110) .. controls (141.43,110.79) and (140.79,111.43) .. (140,111.43) .. controls (139.21,111.43) and (138.57,110.79) .. (138.57,110) -- cycle ;
%Shape: Circle [id:dp8336504633373785] 
\draw  [color={rgb, 255:red, 0; green, 0; blue, 255 }  ,draw opacity=0.67 ][fill={rgb, 255:red, 0; green, 0; blue, 255 }  ,fill opacity=0.67 ] (148.6,99.84) .. controls (148.6,99.05) and (149.24,98.41) .. (150.02,98.41) .. controls (150.81,98.41) and (151.45,99.05) .. (151.45,99.84) .. controls (151.45,100.62) and (150.81,101.26) .. (150.02,101.26) .. controls (149.24,101.26) and (148.6,100.62) .. (148.6,99.84) -- cycle ;
%Shape: Circle [id:dp42778300865820373] 
\draw  [color={rgb, 255:red, 0; green, 0; blue, 255 }  ,draw opacity=0.67 ][fill={rgb, 255:red, 0; green, 0; blue, 255 }  ,fill opacity=0.67 ] (148.44,120.14) .. controls (148.44,119.35) and (149.08,118.71) .. (149.87,118.71) .. controls (150.66,118.71) and (151.3,119.35) .. (151.3,120.14) .. controls (151.3,120.93) and (150.66,121.57) .. (149.87,121.57) .. controls (149.08,121.57) and (148.44,120.93) .. (148.44,120.14) -- cycle ;
%Shape: Arc [id:dp16704580702535998] 
\draw  [draw opacity=0] (30,126.5) .. controls (30,126.5) and (30,126.5) .. (30,126.5) .. controls (81.64,126.5) and (123.5,168.36) .. (123.5,220) -- (30,220) -- cycle ; \draw  [color={rgb, 255:red, 144; green, 19; blue, 254 }  ,draw opacity=1 ] (30,126.5) .. controls (30,126.5) and (30,126.5) .. (30,126.5) .. controls (81.64,126.5) and (123.5,168.36) .. (123.5,220) ;  
%Shape: Boxed Polygon [id:dp23012062541722922] 
\draw  [draw opacity=0][fill={rgb, 255:red, 189; green, 16; blue, 224 }  ,fill opacity=0.2 ] (189.83,1.17) -- (244,1) -- (244,51.4) -- (103.62,162.63) -- (100.89,158.78) -- (98.22,155.44) -- (94.67,152.11) -- (91.33,149.22) -- (88.22,147) -- (84.87,144.63) -- cycle ;
%Shape: Polygon [id:ds5333203732636134] 
\draw  [draw opacity=0][fill={rgb, 255:red, 74; green, 144; blue, 226 }  ,fill opacity=0.35 ] (243.6,0.75) -- (243.6,51.15) -- (169.6,109.75) -- (150,109.75) -- (29.6,109.75) -- (29.6,0.75) -- cycle ;
%Straight Lines [id:da9562278484142306] 
\draw    (30,110) -- (139.3,110) -- (160.1,110) -- (160.7,110) -- (250,110) ;
%Straight Lines [id:da29326944749012274] 
\draw    (150,3.2) -- (150,97.6) -- (150,100.8) -- (150,108.5) -- (150,111.5) -- (150,220.2) ;
%Straight Lines [id:da5980967348428912] 
\draw    (30,220) -- (87.08,160.05) -- (121.52,123.88) -- (237.82,1.85) ;
\draw [shift={(239.2,0.4)}, rotate = 133.62] [color={rgb, 255:red, 0; green, 0; blue, 0 }  ][line width=0.75]    (10.93,-3.29) .. controls (6.95,-1.4) and (3.31,-0.3) .. (0,0) .. controls (3.31,0.3) and (6.95,1.4) .. (10.93,3.29)   ;
%Straight Lines [id:da0508170119494713] 
\draw [color={rgb, 255:red, 208; green, 2; blue, 27 }  ,draw opacity=1 ]   (30,220) -- (63.56,174.06) -- (90.4,137.3) -- (187.32,4.62) ;
\draw [shift={(188.5,3)}, rotate = 126.14] [color={rgb, 255:red, 208; green, 2; blue, 27 }  ,draw opacity=1 ][line width=0.75]    (10.93,-3.29) .. controls (6.95,-1.4) and (3.31,-0.3) .. (0,0) .. controls (3.31,0.3) and (6.95,1.4) .. (10.93,3.29)   ;
%Straight Lines [id:da8606678258133422] 
\draw [color={rgb, 255:red, 0; green, 119; blue, 255 }  ,draw opacity=1 ]   (30,220) -- (129.52,141.74) -- (244.93,50.99) ;
\draw [shift={(246.5,49.75)}, rotate = 141.82] [color={rgb, 255:red, 0; green, 119; blue, 255 }  ,draw opacity=1 ][line width=0.75]    (10.93,-3.29) .. controls (6.95,-1.4) and (3.31,-0.3) .. (0,0) .. controls (3.31,0.3) and (6.95,1.4) .. (10.93,3.29)   ;

% Text Node
\draw (9.5,101.5) node [anchor=north west][inner sep=0.75pt]   [align=left] {$N_2$};
% Text Node
\draw (145.5,222.75) node [anchor=north west][inner sep=0.75pt]   [align=left] {$N_1$};
% Text Node
\draw (193, 2) node [anchor=north west][inner sep=0.75pt]  [color={rgb, 255:red, 208; green, 2; blue, 27 }  ,opacity=1 ] [align=left] {$\theta_1$};
% Text Node
\draw (232,62.7) node [anchor=north west][inner sep=0.75pt]  [color={rgb, 255:red, 0; green, 118; blue, 252 }  ,opacity=1 ] [align=left] {$\theta_2$};
% Text Node
\draw (231.7,10) node [anchor=north west][inner sep=0.75pt]   [align=left] {$\theta$};
% Text Node
\draw (64.5,120) node [anchor=north west][inner sep=0.75pt]  [font=\small,color={rgb, 255:red, 144; green, 19; blue, 254 }  ,opacity=1 ]  {$r$};
\end{tikzpicture}
\caption{The $(b_1, b_2)$-plane corresponding to the proof of \autoref{thm:error-dim-3}. The purple sector region shows where the ratio bounds in \eqref{eq:ratio-less-mu} hold; The red region and the blue region correspond to \autoref{claim:term-ratio} (1) and (2), respectively.}
\label{fig:sector-region}
\end{figure}

Since $f_1, f_2$ are rational functions, there exist $\theta_1, \theta_2$ with $\theta_2 < \theta_0 < \theta_1$ such that
\begin{align*}
    \forall\; r \geq r_0 \text{ and } \theta\in [\theta_2, \theta_1], \quad \left|\beta_i f_i(r\cos\theta, r\sin\theta)\right| \leq (1-\lambda_{\min})(1+\epsilon) \; \text{ for } i =  1, 2.
\end{align*}
Let $S$ be the sector whose polar coordinates $r, \theta$ satisfy 
\[r>r_0 \quad \text{ and } \quad \theta\in  [\theta_2, \theta_1].\] 
Denote
\begin{equation}\label{eq:mu-def}
 \mu := (1-\lambda_{\min})(1+\epsilon) =\sqrt{(1-\lambda_{\min})\rho} < 1.
\end{equation}
We note that for $(b_1, b_2)\in \R_{+}^2$,
\begin{align*}
    \beta_1 f_1(b_1, b_2) &= \left(A_{b_1+1, b_2}\,\beta_1^{b_1+1}\beta_2^{b_2}\right)/
\left(A_{b_1,b_2}\,\beta_1^{b_1}\beta_2^{b_2}\right);\\
    \beta_2 f_2(b_1, b_2) &= \left(A_{b_1, b_2+1}\,\beta_1^{b_1}\beta_2^{b_2+1}\right)/
\left(A_{b_1,b_2}\,\beta_1^{b_1}\beta_2^{b_2}\right).
\end{align*}
Thus, for any $(b_1, b_2)\in \R_{+}^2$ that lies in the sector $S$,
\begin{equation} \label{eq:ratio-less-mu}
    \frac{\left|A_{b_1+1, b_2}\,\beta_1^{b_1+1}\beta_2^{b_2}\right|}{\left|A_{b_1,b_2}\,\beta_1^{b_1}\beta_2^{b_2}\right|} \leq \mu
    \quad \text{ and } \quad 
    \frac{\left|A_{b_1, b_2+1}\,\beta_1^{b_1}\beta_2^{b_2+1}\right|}{\left|A_{b_1,b_2}\,\beta_1^{b_1}\beta_2^{b_2}\right|} \leq \mu.
\end{equation} 
Let $(N_1, N_2)$ be an integer point in the interior of $S$ such that $(N_1\pm 1, N_2)$ and $(N_1, N_2\pm 1)$ are also inside the sector $S$. Such point $(N_1, N_2)$ exists because the conditions below can always be met by scaling.
\begin{equation}\label{eq:n1-n2}
\begin{split}
  & N_2\cot \theta_1 +1 \leq N_1 \leq N_2\cot \theta_2 -1 \text{ and }\\
  & N_1\tan \theta_2 +1 \leq N_2 \leq N_1\tan \theta_1 -1. 
\end{split}
\end{equation}
We illustrate the parameters and the regions in \autoref{fig:sector-region}.

Using \eqref{eq:ratio-less-mu}--\eqref{eq:n1-n2} together with \autoref{lemma:fi-monotone}, we obtain the following claim.
\begin{claim} \label{claim:term-ratio}
Let $(b_1, b_2)\in \N^2$. 
\begin{enumerate}
    \item If $b_1 \geq N_1$ and $b_2 \leq b_1 \tan \theta_1$, %(i.e.,  $(b_1, b_2)$ lies below the ray with angle $\theta_1$), 
    then
   $\frac{\left|A_{b_1+1, b_2}\,\beta_1^{b_1+1}\beta_2^{b_2}\right|}{\left|A_{b_1,b_2}\,\beta_1^{b_1}\beta_2^{b_2}\right|} \leq \mu.$    
    \item If $b_2 \geq N_2$ and $b_2 \geq b_1 \tan \theta_2$, then 
     $\frac{\left|A_{b_1, b_2+1}\,\beta_1^{b_1}\beta_2^{b_2+1}\right|}{\left|A_{b_1,b_2}\,\beta_1^{b_1}\beta_2^{b_2}\right|} \leq \mu.$    
\end{enumerate}
\end{claim}

In order to study the truncation errors $E(N_1, N_2)$ and $E(N_1+\ell, N_2+\ell)$, we rewrite them as 
\begin{equation}\label{eq:E-s}
    E(N_1, N_2) = S_1 + S_2 + S_3 
\end{equation}
and
\begin{equation}\label{eq:E-s-l}
    E(N_1+ \ell, N_2+\ell)  \leq S_1' + S_2' + S_3' + S_4',
\end{equation}
respectively, where 
\begin{align*}
    S_1 & = \sum_{\substack{b_1\geq N_1\\ 0\leq b_2 < N_2}} \left|A_{b_1, b_2}\beta_1^{b_1}\beta_2^{b_2}\right|; \;
    & S_2 = \sum_{\substack{b_2\geq N_2\\ 0\leq b_1 < N_1}} \left|A_{b_1, b_2}\beta_1^{b_1}\beta_2^{b_2}\right|; \\
    S_3 & = \sum_{\substack{b_1\geq N_1\\ b_2\geq N_2}} \left|A_{b_1, b_2}\beta_1^{b_1}\beta_2^{b_2}\right|,
\end{align*}
and
\begin{align*}
    S_1' & = \sum_{\substack{b_1\geq N_1+\ell\\ 0\leq b_2 < N_2}} \left|A_{b_1, b_2}\beta_1^{b_1}\beta_2^{b_2}\right|; \;
    & S_2'  = \sum_{\substack{b_2\geq N_2+\ell\\ 0\leq b_1 < N_1}} \left|A_{b_1, b_2}\beta_1^{b_1}\beta_2^{b_2}\right|; \\
    S_3' & = \sum_{\substack{b_1\geq N_1\\ b_2\geq N_2+\ell}} \left|A_{b_1, b_2}\beta_1^{b_1}\beta_2^{b_2}\right|; \;
    & S_4' = \sum_{\substack{b_1\geq N_1+\ell\\ b_2\geq N_2}} \left|A_{b_1, b_2}\beta_1^{b_1}\beta_2^{b_2}\right|.
\end{align*}
We view each term $|A_{b_1, b_2}\beta_1^{b_1}\beta_2^{b_2}|$ in a series $S_k$ as corresponding to the point $(b_1, b_2)$ in the first quadrant. See \autoref{fig:E-N-partition}. %for an illustration.

\begin{figure}[tp]
\centering
\begin{minipage}{.48\textwidth}
\tikzset{every picture/.style={line width=0.75pt}} %set default line width to 0.75pt    
\begin{tikzpicture}[x=0.75pt,y=0.75pt,yscale=-0.6,xscale=0.6]
%Shape: Axis 2D [id:dp8460068849625253] 
\draw  (43,259.67) -- (347,259.67)(70.33,10) -- (70.33,286) (340,254.67) -- (347,259.67) -- (340,264.67) (65.33,17) -- (70.33,10) -- (75.33,17)  ;
%Straight Lines [id:da08897938934349336] 
\draw    (209.5,9.75) -- (210.5,260.75) ;
%Straight Lines [id:da4724087373962367] 
\draw    (350.67,139.5) -- (67.87,139.53) ;
%Straight Lines [id:da3588811442916606] 
\draw [color={rgb, 255:red, 208; green, 2; blue, 27 }  ,draw opacity=1 ]   (70.33,259.67) -- (274.73,10.3) ;
\draw [shift={(276,8.75)}, rotate = 489.34] [color={rgb, 255:red, 208; green, 2; blue, 27 }  ,draw opacity=1 ][line width=0.75]    (10.93,-3.29) .. controls (6.95,-1.4) and (3.31,-0.3) .. (0,0) .. controls (3.31,0.3) and (6.95,1.4) .. (10.93,3.29)   ;
%Shape: Polygon [id:ds06282619575296278] 
\draw  [draw opacity=0][fill={rgb, 255:red, 130; green, 251; blue, 114 }  ,fill opacity=0.39 ] (70.2,19.4) -- (197,19.4) -- (197,139) -- (70.2,139) -- cycle ;
%Shape: Polygon [id:ds9348402789825518] 
\draw  [draw opacity=0][fill={rgb, 255:red, 114; green, 251; blue, 232 }  ,fill opacity=0.39 ] (209,149.83) -- (340,149.83) -- (340,259.17) -- (209,260.75) -- cycle ;
%Shape: Polygon [id:ds3723586923462623] 
\draw  [draw opacity=0][fill={rgb, 255:red, 237; green, 158; blue, 233 }  ,fill opacity=0.39 ] (209,18.75) -- (341.5,19.25) -- (340.5,140) -- (209,140) -- cycle ;
% rays
\draw [color={rgb, 255:red, 74; green, 144; blue, 226 }  ,draw opacity=1 ]   (70.33,259.67) -- (342.86,68.4) ;
\draw [shift={(344.5,67.25)}, rotate = 504.94] [color={rgb, 255:red, 74; green, 144; blue, 226 }  ,draw opacity=1 ][line width=0.75]    (10.93,-3.29) .. controls (6.95,-1.4) and (3.31,-0.3) .. (0,0) .. controls (3.31,0.3) and (6.95,1.4) .. (10.93,3.29)   ;
%Shape: Ellipse [id:dp3897711514919231] 
\draw  [color={rgb, 255:red, 63; green, 13; blue, 239 }  ,draw opacity=1 ][fill={rgb, 255:red, 49; green, 17; blue, 230 }  ,fill opacity=1 ] (207.58,138.87) .. controls (207.58,137.79) and (208.52,136.92) .. (209.67,136.92) .. controls (210.82,136.92) and (211.75,137.79) .. (211.75,138.87) .. controls (211.75,139.96) and (210.82,140.83) .. (209.67,140.83) .. controls (208.52,140.83) and (207.58,139.96) .. (207.58,138.87) -- cycle ;

% Text Node
%\draw (59.33,255.73) node [anchor=north west][inner sep=0.75pt]    {$o$};
% Text Node
\draw (201.33,265.73) node [anchor=north west][inner sep=0.75pt]    {$N_{1}$}; 
% Text Node
\draw (33.5,128.07) node [anchor=north west][inner sep=0.75pt]    {$N_{2}$};
% Text Node
\draw (278,3.65) node [anchor=north west][inner sep=0.75pt]  [color={rgb, 255:red, 208; green, 2; blue, 27 }  ,opacity=1 ]  {$\theta_1$};
% Text Node
\draw (352.5,51.9) node [anchor=north west][inner sep=0.75pt]  [color={rgb, 255:red, 245; green, 166; blue, 35 }  ,opacity=1 ]  {$\textcolor[rgb]{0.29,0.56,0.89}{\theta_2}$};
% Text Node
\draw (210.39,120.18) node [anchor=north west][inner sep=0.75pt]  [color={rgb, 255:red, 82; green, 10; blue, 227 }  ,opacity=1 ]  {};
% Text Node
\draw (128.37,52.7) node [anchor=north west][inner sep=0.75pt]  [color={rgb, 255:red, 208; green, 2; blue, 27 }  ,opacity=1 ]  {$\textcolor[rgb]{0.25,0.46,0.02}{S}\textcolor[rgb]{0.25,0.46,0.02}{_{2}}\textcolor[rgb]{0.25,0.46,0.02}{\ }$};
% Text Node
\draw (277.67,189.9) node [anchor=north west][inner sep=0.75pt]  [color={rgb, 255:red, 208; green, 2; blue, 27 }  ,opacity=1 ]  {$\textcolor[rgb]{0,0.86,0.74}{S}\textcolor[rgb]{0,0.86,0.74}{_{1}}$};
% Text Node
\draw (279.67,52.4) node [anchor=north west][inner sep=0.75pt]  [color={rgb, 255:red, 208; green, 2; blue, 27 }  ,opacity=1 ]  {$\textcolor[rgb]{0.85,0.33,0.77}{S}\textcolor[rgb]{0.85,0.33,0.77}{_{3}}$};
% Text Node
\draw (210.39,120.18) node [anchor=north west][inner sep=0.75pt]  [color={rgb, 255:red, 82; green, 10; blue, 227 }  ,opacity=1 ]  {$s$};
\end{tikzpicture}
\end{minipage}
\begin{minipage}{.48\textwidth}
\tikzset{every picture/.style={line width=0.75pt}} %set default line width to 0.75pt    
\begin{tikzpicture}[x=0.75pt,y=0.75pt,yscale=-0.6,xscale=0.6]
%Shape: Axis 2D [id:dp8460068849625253] 
\draw  (43,259.67) -- (347,259.67)(70.33,10) -- (70.33,286) (340,254.67) -- (347,259.67) -- (340,264.67) (65.33,17) -- (70.33,10) -- (75.33,17)  ;
%Straight Lines [id:da08897938934349336] 
\draw    (209.5,9.75) -- (210.5,260.75) ;
%Straight Lines [id:da4724087373962367] 
\draw    (350.67,139.5) -- (67.87,139.53) ;
%Straight Lines [id:da3588811442916606] 
\draw [color={rgb, 255:red, 208; green, 2; blue, 27 }  ,draw opacity=1 ]   (70.33,259.67) -- (274.73,10.3) ;
\draw [shift={(276,8.75)}, rotate = 489.34] [color={rgb, 255:red, 208; green, 2; blue, 27 }  ,draw opacity=1 ][line width=0.75]    (10.93,-3.29) .. controls (6.95,-1.4) and (3.31,-0.3) .. (0,0) .. controls (3.31,0.3) and (6.95,1.4) .. (10.93,3.29)   ;
%Shape: Polygon [id:ds06282619575296278] 
\draw  [draw opacity=0][fill={rgb, 255:red, 130; green, 251; blue, 114 }  ,fill opacity=0.39 ] (70.2,19.4) -- (197,20.17) -- (196.33,79.5) -- (70.87,79.53) -- cycle ;
%Shape: Polygon [id:ds9348402789825518] 
\draw  [draw opacity=0][fill={rgb, 255:red, 114; green, 251; blue, 232 }  ,fill opacity=0.39 ] (270.33,149.83) -- (350.67,149.83) -- (349.67,259.17) -- (270.5,260.75) -- cycle ;
%Shape: Polygon [id:ds9172854524069174] 
\draw  [draw opacity=0][fill={rgb, 255:red, 208; green, 2; blue, 27 }  ,fill opacity=0.32 ] (209.8,19.8) -- (351,19) -- (350,78) -- (209,78.75) -- cycle ;
%Straight Lines [id:da3689381731004908] 
\draw    (350,78) -- (70.87,79.53) ;
%Straight Lines [id:da8257754795416596] 
\draw    (269.5,9.75) -- (270.5,260.75) ;
%Shape: Polygon [id:ds340141105530932] 
\draw  [draw opacity=0][fill={rgb, 255:red, 248; green, 231; blue, 28 }  ,fill opacity=0.54 ] (269.2,19.4) -- (350.4,18.2) -- (350.67,139.5) -- (270.67,139.83) -- cycle ;
%Straight Lines [id:da8888216043763069] 
\draw [color={rgb, 255:red, 74; green, 144; blue, 226 }  ,draw opacity=1 ]   (70.33,259.67) -- (342.86,68.4) ;
\draw [shift={(344.5,67.25)}, rotate = 504.94] [color={rgb, 255:red, 74; green, 144; blue, 226 }  ,draw opacity=1 ][line width=0.75]    (10.93,-3.29) .. controls (6.95,-1.4) and (3.31,-0.3) .. (0,0) .. controls (3.31,0.3) and (6.95,1.4) .. (10.93,3.29)   ;
%Shape: Ellipse [id:dp3897711514919231] 
\draw  [color={rgb, 255:red, 63; green, 13; blue, 239 }  ,draw opacity=1 ][fill={rgb, 255:red, 49; green, 17; blue, 230 }  ,fill opacity=1 ] (207.58,138.87) .. controls (207.58,137.79) and (208.52,136.92) .. (209.67,136.92) .. controls (210.82,136.92) and (211.75,137.79) .. (211.75,138.87) .. controls (211.75,139.96) and (210.82,140.83) .. (209.67,140.83) .. controls (208.52,140.83) and (207.58,139.96) .. (207.58,138.87) -- cycle ;

% Text Node
%\draw (59.33,255.73) node [anchor=north west][inner sep=0.75pt]    {$o$};
% Text Node
\draw (201.33,265.73) node [anchor=north west][inner sep=0.75pt]    {$N_{1}$}; 
% Text Node
\draw (33.5,128.07) node [anchor=north west][inner sep=0.75pt]    {$N_{2}$};
% Text Node
\draw (278,3.65) node [anchor=north west][inner sep=0.75pt]  [color={rgb, 255:red, 208; green, 2; blue, 27 }  ,opacity=1 ]  {$\theta_1$};
% Text Node
\draw (352.5,51.9) node [anchor=north west][inner sep=0.75pt]  [color={rgb, 255:red, 245; green, 166; blue, 35 }  ,opacity=1 ]  {$\textcolor[rgb]{0.29,0.56,0.89}{\theta_2}$};
% Text Node
\draw (112.7,41.03) node [anchor=north west][inner sep=0.75pt]  [color={rgb, 255:red, 208; green, 2; blue, 27 }  ,opacity=1 ]  {$\textcolor[rgb]{0.25,0.46,0.02}{S}\textcolor[rgb]{0.25,0.46,0.02}{_{2}}\textcolor[rgb]{0.25,0.46,0.02}{'\ }$};
% Text Node
\draw (300.67,194.9) node [anchor=north west][inner sep=0.75pt]  [color={rgb, 255:red, 208; green, 2; blue, 27 }  ,opacity=1 ]  {$\textcolor[rgb]{0,0.86,0.74}{S}\textcolor[rgb]{0,0.86,0.74}{_{1}}\textcolor[rgb]{0,0.86,0.74}{'}$};
% Text Node
\draw (223.67,26.4) node [anchor=north west][inner sep=0.75pt]  [color={rgb, 255:red, 208; green, 2; blue, 27 }  ,opacity=1 ]  {$\textcolor[rgb]{0.33,0.04,0.07}{S_{3} '}$};
% Text Node
\draw (210.39,120.18) node [anchor=north west][inner sep=0.75pt]  [color={rgb, 255:red, 82; green, 10; blue, 227 }  ,opacity=1 ]  {};
% Text Node
\draw (0,68.07) node [anchor=north west][inner sep=0.75pt]    {$N_{2}+\ell $};
% Text Node
\draw (258,266.07) node [anchor=north west][inner sep=0.75pt]    {$N_{1}+\ell $};
% Text Node
\draw (311.37,103.7) node [anchor=north west][inner sep=0.75pt]  [color={rgb, 255:red, 119; green, 111; blue, 11 }  ,opacity=1 ]  {$\textcolor[rgb]{0.42,0.39,0.04}{S_{4} '}\textcolor[rgb]{0.12,0.44,0.81}{\ }$};
% Text Node
\draw (210.39,120.18) node [anchor=north west][inner sep=0.75pt]  [color={rgb, 255:red, 82; green, 10; blue, 227 }  ,opacity=1 ]  {$s$};
\end{tikzpicture}
\end{minipage}
\caption{Planes of multiexponents $(b_1, b_2)\in \N^2$ for analysing the series truncation errors $E(N_1, N_2)$ (left) and $E(N_1+\ell, N_2+\ell)$ (right). The color of a point $(b_1, b_2)$ indicates which series $S_k$ the term $|A_{b_1, b_2}\beta_1^{b_1}\beta_2^{b_2}|$ belongs to in equations~\eqref{eq:E-s} and \eqref{eq:E-s-l}.}
\label{fig:E-N-partition}
\end{figure}
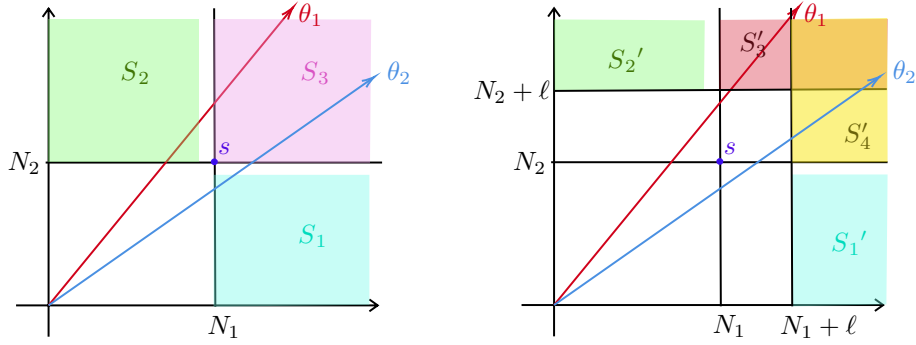

We first consider $S_1$ and $S_1'$. Let $(b_1, b_2)\in \N^2$ such that $b_1 = N_1 + i$ for some $i\geq 0$ and $b_2 \leq N_2$. We compare the term corresponding to  $(b_1, b_2)$ with the term corresponding to $(N_1, b_2)$.
Using \autoref{claim:term-ratio}--(1), we have that 
\begin{equation*} %\label{eq:ratio-compare-mu-i}
    \left|\frac{A_{N_1+i, b_2}\beta_1^{N_1+i}\beta_2^{b_2}}{A_{N_1,b_2}\beta_1^{N_1}\beta_2^{b_2}}\right| = \prod_{j=0}^{i-1} {\left|\frac{A_{N_1+j+1, b_2}\beta_1^{N_1+j+1}\beta_2^{b_2}}{A_{N_1+j,b_2}\beta_1^{N_1+j}\beta_2^{b_2}}\right|} \leq \mu^i.
\end{equation*}
In particular, the above inequality holds for $i=\ell$ and any $0\leq b_2 \leq N_2-1$.
This implies that 
\begin{equation}\label{eq:s1-bound}
    S_1' \leq \mu^\ell S_1.
\end{equation}
Similarly, using \autoref{claim:term-ratio}--(2), we obtain that 
\begin{equation}\label{eq:s2-bound}
    S_2' \leq \mu^\ell S_2.
\end{equation}

Next, we compare a term corresponding to $(b_1, b_2)$ in $S_3$ to the term $s$ corresponding to $(N_1, N_2)$, which is defined by 
\begin{equation*}
    s := \left|A_{N_1, N_2}\beta_1^{N_1}\beta_2^{N_2}\right|.
\end{equation*}
\begin{claim}\label{claim:term-ratio-s3}
   Let $b_1 = N_1+i$ and $b_2=N_2+j$, where $i,j$ are non-negative integers. Then 
   \begin{equation*}
       \frac{\left|A_{b_1, b_2}\beta_1^{b_1}\beta_2^{b_2}\right|}{s}\leq \mu^{i+j}.
   \end{equation*}
\end{claim}
\begin{proof}
   Depending on whether the point $(b_1, b_2)$ lies below or above the ray from the origin through the point $(N_1, N_2)$, we reduce either $b_1$ or $b_2$ by $1$, as follows. By the geometric conditions in \eqref{eq:n1-n2}, if $(b_1, b_2)$ is below the ray, then $b_1-1$ lies in the sector $S$, and if $(b_1, b_2)$ is above the ray, then $b_2-1$ lies in the sector $S$. In either case, we can apply \autoref{claim:term-ratio} to obtain that the ratio between the terms corresponding to $(b_1, b_2)$ and $(b_1', b_2')=(b_1, b_2-1)\text{ or }(b_1-1, b_2)$ is at most $\mu$. Repeat this process until $(b_1', b_2')=(N_1, N_2)$. By multiplying the sequences of ratios obtained in the process, we have the desired inequality.
\end{proof}

In particular, \autoref{claim:term-ratio-s3} provides upper bounds of $S_3'$ and $S_4'$ in terms of $s$.
\begin{equation}\label{eq:s34p-bound}
\begin{split}
    S_3' \leq s \sum_{i=0}^{\infty} \sum_{j=\ell}^{\infty} \mu^{i+j} = \frac{\mu^{\ell} s}{(1-\mu)^2}, \\
    S_4' \leq s \sum_{i=\ell}^{\infty} \sum_{j=0}^{\infty} \mu^{i+j} = \frac{\mu^{\ell} s}{(1-\mu)^2}.
\end{split}
\end{equation}
\begin{claim}
    For very large $N_1, N_2$, we may assume that $s / S_1$ is sufficiently small.
\end{claim}
\begin{proof}
    Let $(N^\ast_1, N^\ast_2)\in \N^2$ be a point in the sector $S$ such that $N^\ast_1\geq N_1, N^\ast_2= N_2+k$ and $k := \floor{N^\ast_2 - N^\ast_1 \tan\theta_2}$ is a large integer. Define $s^\ast$ to be the term corresponding to the integer point $(N^\ast_1, N^\ast_2)$, and define the series $S^\ast_1$ accordingly. Notice that we can lower bound $S^\ast_1$ by the single term $t$ inside it, where $t := \left|A_{N^\ast_1, N_2}\beta_1^{N^\ast_1}\beta_2^{N_2}\right|$ corresponding to the point $(N^\ast_1, N_2)$. Since this point lies in $S$ and $N_2 = N^\ast_2-k$, \autoref{claim:term-ratio}--(2) implies that $s^\ast / t \leq \mu^k$. Therefore, $s^\ast / S^\ast_1 \leq \mu^k$, which is sufficiently small for $k$ large enough.
\end{proof}
In particular, we take large $N_1, N_2$ such that
$s / S_1 \leq \tfrac{1}{2}\epsilon (1-\mu)^2$,
or equivalently,
\begin{equation}\label{eq:s-eps-s1}
    \frac{2s}{(1-\mu)^2} \leq \epsilon S_1 .
\end{equation}
%\tred{Don't have $S_3'\leq S_3 \mu^\ell$  and $S_4'\leq S_3 \mu^\ell$.}
By combining \eqref{eq:E-s}--\eqref{eq:s-eps-s1}, %\eqref{eq:E-s}, \eqref{eq:E-s-l}, \eqref{eq:s1-bound}, \eqref{eq:s2-bound}, \eqref{eq:s34p-bound} and \eqref{eq:s-eps-s1}, 
we obtain that
\begin{align*}
     E(N_1+ \ell, N_2+\ell) 
    & \leq S_1' + S_2' + S_3' + S_4' \leq \mu^\ell\left(S_1+S_2+\frac{2s}{(1-\mu)^2}\right)\\
    & \leq \mu^\ell\left(S_1+S_2+\epsilon S_1\right) \leq \mu^\ell (1+\epsilon) (S_1+S_2+S_3)\\
    & \leq \mu^\ell (1+\epsilon)^\ell (S_1+S_2+S_3) \\
    & = (\mu(1+\epsilon))^\ell E(N_1, N_2).
\end{align*}
Finally, according to the definitions~\eqref{eq:eps-def} and \eqref{eq:mu-def}, 
\[\mu(1+\epsilon) = \sqrt{(1-\lambda_{\min})\rho} \cdot \sqrt{\frac{\rho}{1-\lambda_{\min}}} =\rho.\]
Therefore, the desired inequality $E(N_1+ \ell, N_2+\ell)\leq \rho^{\ell} E(N_1, N_2)$ holds.
\end{proof}

\clearpage

%{\small
\providecommand\ISBN{ISBN }
\bibliographystyle{amsabbrvurl}
\bibliography{ref.bib}
%}

\end{document}